\documentclass[11pt,twoside,english]{article}
\usepackage{lmodern}
\usepackage[T1]{fontenc}
\usepackage[latin9]{inputenc}
\usepackage{geometry}
\usepackage{color}
\usepackage{babel}
\usepackage{array}
\usepackage{url}
\usepackage{multirow}
\usepackage{amsmath}
\usepackage{amsthm}
\usepackage{amssymb}
\usepackage{graphicx}
\usepackage{tablefootnote}
\usepackage[unicode=true,
 bookmarks=false,
 breaklinks=false,pdfborder={0 0 1},backref=section,colorlinks=true]
 {hyperref}

\makeatletter

\providecommand{\tabularnewline}{\\}

\theoremstyle{plain}
\newtheorem{thm}{\protect\theoremname}
\theoremstyle{plain}
\newtheorem{prop}[thm]{\protect\propositionname}
\theoremstyle{plain}
\newtheorem{lem}[thm]{\protect\lemmaname}

\usepackage{xcolor}
\usepackage[numbers,sort]{natbib}

\usepackage{svg}
\usepackage{multicol}
\usepackage{enumitem}
\usepackage{mdframed}
\mdfsetup{backgroundcolor=gray!1.5}
\definecolor{green}{rgb}{0,0.6,0.0} 

\usepackage{hyperref}
\hypersetup{
    colorlinks=true,
    linkcolor=green,
    linkbordercolor=red,
    citecolor=blue,
    filecolor=blue,
    urlcolor=purple
}

\makeatother

\providecommand{\lemmaname}{Lemma}
\providecommand{\propositionname}{Proposition}
\providecommand{\theoremname}{Theorem}

\begin{document}
\global\long\def\rn{\mathbb{R}^{n}}%
\global\long\def\r{\mathbb{R}}%
\global\long\def\n{\mathbb{N}}%
\global\long\def\c{\mathbb{C}}%
\global\long\def\pt{\mathbb{\partial}}%
\global\long\def\lam{\lambda}%
\global\long\def\argmin{\operatorname*{argmin}}%
\global\long\def\Argmin{\operatorname*{Argmin}}%
\global\long\def\argmax{\operatorname*{argmax}}%
\global\long\def\dom{\operatorname*{dom}}%
\global\long\def\ri{\operatorname*{ri}}%
\global\long\def\diag{\operatorname*{diag}}%
\global\long\def\Diag{\operatorname*{Diag}}%
\global\long\def\inner#1#2{\langle#1,#2\rangle}%
\global\long\def\dg{\operatorname*{dg}}%
\global\long\def\trc{\operatorname*{tr}}%
\global\long\def\Dg{\operatorname*{Dg}}%
\global\long\def\cConv{\overline{{\rm Conv}}\ }%
\global\long\def\conv{\operatorname*{conv}}%

\global\long\def\acgX#1{z_{#1}}%
\global\long\def\acgXTilde#1{\tilde{z}_{#1}}%
\global\long\def\acgU#1{r_{#1}}%
\global\long\def\acgY#1{z_{#1}^{c}}%
\global\long\def\acgMatX#1{\uppercase{z}_{#1}}%
\global\long\def\acgMatXTilde#1{\tilde{\uppercase{z}}_{#1}}%
\global\long\def\acgMatU#1{\uppercase{r}_{#1}}%
\global\long\def\acgMatY#1{\uppercase{z}_{#1}^{c}}%
\global\long\def\icgY#1{y_{#1}}%
\global\long\def\icgMatY#1{\uppercase{y}_{#1}}%
\global\long\def\aicgY#1{y_{#1}^{a}}%
\global\long\def\aicgYMin#1{y_{#1}}%
\global\long\def\aicgXTilde#1{\tilde{x}_{#1}}%
\global\long\def\aicgXhat#1{\widehat{x}_{#1}}%
\global\long\def\aicgX#1{x_{#1}}%
\global\long\def\mTick{-}%
\global\long\def\pTick{+}%
\global\long\def\gammaBFn{\gamma}%
\global\long\def\gammaBTFn{\widetilde{\gamma}}%
\global\long\def\gammaFn#1{\gammaBFn\left(#1\right)}%
\global\long\def\gammaTFn#1{\gammaBTFn\left(#1\right)}%
\global\long\def\Y{y^{a}}%
\global\long\def\YMin{y}%
\global\long\def\Xt{\tilde{x}}%
\global\long\def\X{x}%
\global\long\def\Xh{\widehat{\X}}%
\global\long\def\YM{\YMin}%
\global\long\def\XtM{\Xt}%
\global\long\def\XM{\X}%
\global\long\def\aM{a}%
\global\long\def\AM{A}%
\global\long\def\DM{D}%
\global\long\def\deltaM{\delta}%
\global\long\def\YP{\Y_{\pTick}}%
\global\long\def\YMinP{\YMin_{\pTick}}%
\global\long\def\XtP{\Xt_{\pTick}}%
\global\long\def\XhP{\Xh_{\pTick}}%
\global\long\def\XP{\X_{\pTick}}%
\global\long\def\AP{A_{\pTick}}%
\global\long\def\thetaM{\pi}%
\global\long\def\thetaP{\pi_{\pTick}}%

\title{Accelerated Inexact Composite Gradient Methods for\\
Nonconvex Spectral Optimization Problems\thanks{The works of these authors     were partially supported by ONR Grant N00014-18-1-2077,
AFOSR Grant FA9550-22-1-0088,
NSERC Grant PGSD3-516700-2018, and the IDEaS-TRIAD Fellowship (NSF Grant CCF-1740776). The first author has been supported by the US Department of Energy (DOE) and UT-Battelle, LLC, under contract DE-AC05-00OR22725 and also supported by the Exascale Computing Project (17-SC-20-SC), a collaborative effort of the U.S. Department of Energy Office of Science and the National Nuclear Security Administration.}}
\author{Weiwei Kong \thanks{Computer Science and Mathematics Division, Oak Ridge National Laboratory, Oak Ridge, TN, 37830 (E-mail: {\tt        wwkong92@gmail.com}). }
	\and Renato D.C. Monteiro\thanks{School of Industrial and Systems     Engineering, Georgia Institute of     Technology, Atlanta, GA, 30332-0205. (E-mail: {\tt monteiro@isye.gatech.edu}). }}
\date{July 22, 2020 (Revised: July 7, 2021 and February 9, 2022)}
\maketitle
\begin{abstract}
	This paper presents two inexact composite gradient methods, one inner
	accelerated and another doubly accelerated, for solving a class of
	nonconvex spectral composite optimization problems. More specifically,
	the objective function for these problems is of the form $f_{1}+f_{2}+h$,
	where $f_{1}$ and $f_{2}$ are differentiable nonconvex matrix functions
	with Lipschitz continuous gradients, $h$ is a proper closed convex
	matrix function, and both $f_{2}$ and $h$ can be expressed as functions
	that operate on the singular values of their inputs. The methods essentially
	use an accelerated composite gradient method to solve a sequence of
	proximal subproblems involving the linear approximation of $f_{1}$
	and the singular value functions underlying $f_{2}$ and $h$. Unlike
	other composite gradient-based methods, the proposed methods take
	advantage of both the composite and spectral structure underlying
	the objective function in order to efficiently generate their solutions.
	Numerical experiments are presented to demonstrate the practicality
	of these methods on a set of real-world and randomly generated spectral
	optimization problems. \\
	
\end{abstract}

\section{\label{sec:intro}Introduction}

\noindent There are numerous applications in electrical engineering,
machine learning, and medical imaging that can be formulated as nonconvex
spectral optimization problems of the form 
\begin{equation}
\min_{U\in\r^{m\times n}}\left\{ \phi(U):=f_{1}(U)+(\underbrace{f_{2}^{{\cal V}}\circ\sigma}_{f_2})(U)+(\underbrace{h^{{\cal V}}\circ\sigma}_h)(U)\right\} ,\label{eq:intro_prb}
\end{equation}
where $\sigma$ is the function that maps a matrix to its singular
value vector (in nonincreasing order of magnitude), $f_{1}$ and $f_{2}^{{\cal V}}$
are continuously differentiable functions with Lipschitz continuous
gradients, and $h^{{\cal V}}$ is a proper, lower semicontinuous,
convex function. For this paper, we are interested in solving instances of \eqref{eq:intro_prb} where: (i) the resolvents of $\lam\pt h$ and $\lam\pt h^{{\cal V}}$,
i.e., evaluations of the operators $(I+\lam\pt h)^{-1}$ and $(I+\lam\pt h^{{\cal V}})^{-1}$,
are easy compute for any $\lam>0$; (ii) the resolvents of $\lam (\nabla f_2 + \partial h)$ and $\lam (\nabla f_2^{\cal V} + \partial h^{\cal V})$ cannot be computed exactly for any $\lam > 0$; and (iii) both $f_{2}^{{\cal V}}$
and $h^{{\cal V}}$ are absolutely symmetric in their arguments, i.e.,
they do not depend on the ordering or the sign of their arguments.

We now describe some practical instances of \eqref{eq:intro_prb} that satisfy all three assumptions above. To avoid repetition, we let ${\cal R}={\cal R}_{s}+{\cal R}_{n}$
and ${\cal P}$ be two sparsity-inducing regularizers, where ${\cal R}_{s}$ and ${\cal P}$ are continuously
differentiable functions with Lipschitz continuous gradients and ${\cal R}_{n}$
is a proper, lower semicontinuous, and convex function. 
\begin{itemize}
    \item \textit{Matrix Completion}. Let $A\in\r^{m\times n}$ be a given data matrix and let $r=\min\{m,n\}$.
Moreover, let $\Omega$ denote a subset of the indices of $A$. The
goal of the general matrix completion problem is to find a low rank
approximation of $A$ that is close to $A$ in some sense. A nonconvex formulation (see, for example, \cite{Yao2017}) of this
problem is 
\[
\min_{X\in\r^{m\times n}}\left\{ \frac{1}{2}\|P_{\Omega}(X-A)\|_{F}^{2}+({\cal R}\circ\sigma)(X)\right\} ,
\]
where $P_{\Omega}$ is the function that zeros out the entries of
its input that are not in $\Omega$. Note that
this problem is a special instance of \eqref{eq:intro_prb}
in which $f_1 = \|P_{\Omega}(\cdot)-A\|_F^2 / 2$, $f_2^{\cal V}={\cal R}_s$, and $h^{\cal V}={\cal R}_n$.
\item \textit{Phase Retrieval}. Given a vector $x\in\rn$, let $x[\omega]$ denote its discrete Fourier
transform for some frequency $\omega$. Moreover, for some unknown
noisy signal $\tilde{x}\in\r^{n}$ and a frequency set $\Omega\subseteq\r_{+}$,
suppose that we are given measurements $\{|\tilde{x}[\omega]|\}_{\omega\in\Omega}$
and vectors $a_{\omega}\in\c^{n}$ such that $|\left\langle a_{\omega},\tilde{x}\right\rangle |=|\tilde{x}[\omega]|$
for every $\omega\in\Omega$. The goal of the phase retrieval problem
is to recover an approximation $x$ of $\tilde{x}$ such that $|\left\langle a_{\omega},x\right\rangle |^{2}\approx|\left\langle a_{\omega},\tilde{x}\right\rangle |^{2}$
for every $\omega\in\Omega$.
A nonconvex formulation of this problem is 
\[
\begin{aligned}\min_{X\in\r^{|\Omega|\times|\Omega|}}\ \left\{ \frac{1}{2}\|{\cal A}(X)-b\|^{2}+({\cal R}\circ\lam)(X):X\succeq0\right\} ,\end{aligned}
\]
where $\lam$ denotes the function that maps matrices to their eigenvalue
vector, $X\succeq0$ means that $X$ is symmetric positive semidefinite,
and the quantities ${\cal A}:\r^{|\Omega|\times|\Omega|}\mapsto\r^{|\Omega|}$
and $b\in\r^{|\Omega|}$ are given by 
\[
\left[{\cal A}(X)\right]_{\omega}=\trc(a_{\omega}a_{\omega}^{*}X),\quad b_{\omega}=|\tilde{x}[\omega]|^{2},\quad\forall(X,\omega)\in\r^{|\Omega|\times|\Omega|}\times\Omega.
\]
Note that
this problem is a special instance of \eqref{eq:intro_prb}
in which  $f_1 = \|{\cal A}(\cdot)-b\|_F^2 / 2$, $f_2^{\cal V}={\cal R}_s$, and $h^{\cal V}={\cal R}_n + \delta_{\mathbb{R}_+^{|\Omega|}}$ where $\delta_{\mathbb{R}_+^{|\Omega|}}$ is the indicator for the nonnegative orthant of $\mathbb{R}^{|\Omega|}$.
It is worth mentioning that this formulation is a generalization of the one in
\cite{Candes2015} where the convex function $\trc X$ is replaced
with the nonconvex function ${\cal R}$.
\item \textit{Robust Principal Component Analysis}. Let $\widehat{M}\in\r^{m\times n}$ be a given data matrix and let
$r=\min\{m,n\}$. The goal of the robust principal component analysis
problem is to find an approximation $M+E$ of $\widehat{M}$ where
$M$ is low-rank and $E$ is sparse.
A nonconvex formulation of this problem is 
\[
\begin{aligned}\min_{M,E\in\r^{m\times n}}\ \left\{ \frac{1}{2}\|\widehat{M}-(M+E)\|_{F}^{2}+({\cal R}\circ\sigma)(M)+{\cal P}(E)\right\} .\end{aligned}
\]
Note that
this problem is a special instance of \eqref{eq:intro_prb}
in which  $f_1 = \|\widehat{M}-[(\cdot)-E]\|_F^2 / 2 + {\cal P}$, $f_2^{\cal V}={\cal R}_s$, and $h^{\cal V}={\cal R}_n$.
It is worth mentioning that this formulation is a instance of the one in \cite{Wen2019}
where more structure is imposed on the functions ${\cal R}$ and ${\cal P}$. 
\end{itemize}

A natural approach for finding approximate stationary points of the above instances is to employ
the \textit{exact} composite gradient (ECG) method
that, when applied to \eqref{eq:intro_prb}, \textit{exactly} solves a sequence of matrix subproblems of the form
\begin{equation}
\min_{U\in\r^{m\times n}}\ \left\{ \tilde{\lam}_k \left[\left\langle \nabla(f_{1}+f_{2})(\icgMatY{k-1}),U\right\rangle +h(U)\right]+\frac{1}{2}\|U-\icgMatY{k-1}\|^{2}_F\right\} ,\label{eq:ecg_prox}
\end{equation}
where $\lam_k >0$  is an appropriately chosen stepsize and the point $\icgMatY{k-1}$ is the previous iterate. 
Its computation primarily consists of computing  a singular value decomposition (SVD) at the point $\tilde{Y}_k := \icgMatY{k-1} - \tilde{\lam}_k \nabla(f_1 + f_2)(\icgMatY{k-1})$ and an evaluation of 
the resolvent of $\tilde{\lam}_k \pt h^{\cal V}$ at $\sigma(\tilde{Y}_k)$.
Accelerated ECG (A-ECG) methods solve subproblems similar to \eqref{eq:ecg_prox} but with $Y_{k-1}$ selected in an accelerated manner. 
Notice that both of these approaches do not exploit the spectral structure in
$f_{2}$.

Our goal in this paper is to develop two efficient \textit{inexact} composite
gradient (ICG) methods that find approximate stationary points of \eqref{eq:intro_prb} by exploiting the spectral structure in \textit{both} $f_2$ and $h$.
Our first prototype, called the static inner accelerated ICG (IA-ICG) method,
\textit{inexactly} solves a sequence of matrix prox subproblems of the form
\begin{align}
\min_{U\in\r^{m\times n}}\ \left\{ \lam_k \left[\left\langle \nabla f_{1}(\icgMatY{k-1}),U\right\rangle +f_{2}(U)+h(U)\right]+\frac{1}{2}\|U-\icgMatY{k-1}\|^{2}_F\right\} \label{eq:icg_mat_prox}
\end{align}
where $\lam_k >0$ is an appropriately chosen stepsize and the point $\icgMatY{k-1}$ is the previous iterate.  
It is shown (see Subsection~\ref{subsec:spectral_exploit}) that
the effort of finding the required inexact solution $\icgMatY k$
of \eqref{eq:icg_mat_prox} consists of computing one SVD and applying an accelerated gradient (ACG) algorithm to find an approximate solution
to the related vector prox subproblem 
\begin{equation}
\min_{u\in\r^{r}}\left\{ \lam_k \left[f_{2}^{{\cal V}}(u)-\left\langle c_{k-1},u\right\rangle +h^{{\cal V}}(u)\right]+\frac{1}{2}\|u\|^{2}\right\} \label{eq:icg_vec_prox}
\end{equation}
where $r=\min\{m,n\}$ and $c_{k-1}=\sigma(\icgMatY{k-1}-\lam_k \nabla f_{1}(\icgMatY{k-1}))$.
Notice that \eqref{eq:icg_vec_prox} is a problem over the vector space
$\r^{r}$, and hence, has significantly fewer dimensions than \eqref{eq:icg_mat_prox}
which is a problem over the matrix space $\r^{m\times n}$. The other
prototype, called the static doubly accelerated ICG (DA-ICG), solves
a subproblem similar to \eqref{eq:icg_mat_prox} but with
$\icgMatY{k-1}$ selected in an accelerated manner (and hence its
qualifier ``doubly accelerated''). 
Notice that the static IA-ICG (resp. DA-ICG) can be viewed as an inexact
version of ECG (resp. A-ECG) where, instead of $h$ in \eqref{eq:ecg_prox}, the function $f_2 + h$ is viewed as the composite term, i.e., the part that is not linearized in the subproblems.
Moreover, neither IA-ICG nor DA-ICG are able to 
 solve \eqref{eq:icg_mat_prox} (or its accelerated version) exactly due to assumption (ii) made in the first 
paragraph of this section. \\

\emph{Motivation of our approach}. 
For high-dimensional instances of \eqref{eq:intro_prb}
where $r=\min\{m,n\}$ is large, 
we have that the larger the Lipschitz constant of $\nabla f_{2}^{{\cal V}}$
is, the better the performance of the ICG methods is compared to the performance of their exact counterparts.
This fact immediately follows from the following two claims:
\begin{itemize}
\item[(i)] the ICG methods inexactly solve fewer matrix subproblems compared to their exact counterparts when the Lipschitz constant of $\nabla f_{2}^{{\cal V}}$ is large; and
\item[(ii)] the work of exactly
solving \eqref{eq:ecg_prox} or inexactly solving \eqref{eq:icg_mat_prox} is comparable when $r$ is large.
\end{itemize}

The justification of claim (i) is as follows.
First, recall that the
larger the stepizes $\lam_k$'s (resp. $\tilde{\lam}_k$)
are, the smaller the number of generated subproblems
\eqref{eq:icg_mat_prox} (resp. \eqref{eq:ecg_prox}) is.
Second, the CG stepsizes chosen in either \eqref{eq:ecg_prox} or \eqref{eq:icg_mat_prox} 
to guarantee
convergence of the underlying CG method are inversely proportional to the Lipschitz constant of the gradient of the function being linearized. 
Hence, since the inexact CG methods linearize $f_1$ only and the exact CG methods linearize both $f_1$ and
$f_2$, claim (i) follows.  Some specific
applications where the Lipschitz constant of $\nabla f_{2}^{{\cal V}}$
may be large in practice can be found, for example,
in \cite{ahn2017difference, Yao2017, wen2018proximal}.
The justification for claim (ii) is due to 
the following two observations:
(a) all of the above CG methods require one SVD per subproblem; and 
(b) when $r$ is large, the computational bottleneck for solving a single subproblem is the aforementioned SVD.
\\

\textit{Contributions and Main results.} To the best of our knowledge, this paper is the first to present ICG methods that exploit both the spectral and
composite structure in \eqref{eq:intro_prb}. 

When $f_2$ is convex or, more generally, a key inequality is satisfied at every iteration of ACG applied to \eqref{eq:icg_vec_prox}, it is shown that for any given $\hat{\rho}>0$, both
the static IA-ICG and the static DA-ICG always obtain a pair $(\hat{\icgMatY{}},\hat{V})$
satisfying the approximate stationarity condition
\begin{equation}
    \hat{V}\in\nabla f_{1}(\hat{\icgMatY{}})+\nabla f_{2}(\hat{\icgMatY{}})+\pt h(\hat{\icgMatY{}}),\quad\|\hat{V}\|\leq\hat{\rho}. \label{eq:intro_rho_approx}
\end{equation}
by inexactly solving
${\cal O}(\hat{\rho}^{-2})$ matrix prox subproblems
as in \eqref{eq:icg_mat_prox}. If, in addition,
$f_1$ is convex, 
it is shown that this bound improves to ${\cal O}(\hat{\rho}^{-2/3})$ for the static  DA-ICG method.

When $f_2$ is nonconvex, the static IA-ICG and the static DA-ICG may fail to obtain a pair
as in \eqref{eq:intro_rho_approx}. To remedy this, we develop dynamic IA-ICG and DA-ICG methods that repeatedly invoke their static counterparts to solve \eqref{eq:intro_prb} with $(f_1,f_2)$ replaced by $(f_{1,\xi},f_{2,\xi})=(f_1 - \xi\|\cdot\|^2/2, f_2 + \xi\|\cdot\|^2/2)$ for strictly increasing values of $\xi>0$. 
These dynamic versions always obtain a pair as in \eqref{eq:intro_rho_approx} because: (i) $f_1 + f_2 = f_{1,\xi} + f_{2,\xi}$ for every $\xi>0$ and (ii) there always exists $\underline{\xi}>0$ such that $f_{2,\underline{\xi}}$ is convex due to the fact that $\nabla f_2$ is Lipschitz continuous.

Numerical experiments are also given to demonstrate the practicality
of our proposed methods. More specifically, our experiments demonstrate
that the dynamic methods are substantially faster (usually 10x) than other
first-order methods at minimizing the primal residual $\|\hat{V}||$
in terms of runtime.
\\


\emph{Related works}. The earliest complexity analysis of an ACG method
for solving nonconvex composite problems like the one in \eqref{eq:intro_prb}
is given in \cite{nonconv_lan16}. Building on the results in \cite{nonconv_lan16},
many other papers \cite{Paquette2017,LanUniformly,Liang2019} have
proposed similar ACG-based methods.

Another common approach for solving problems like \eqref{eq:intro_prb}
is to employ an inexact proximal point method where each prox subproblem
is constructed to be convex, and hence, solvable by an ACG variant.
For example, papers \cite{Aaronetal2017,rohtua,WJRproxmet1,WJRVarLam2018}
present inner accelerated inexact proximal point methods whereas \cite{Liang2018}
presents a doubly accelerated inexact proximal point method.\\

\emph{Organization of the paper.} Subsection~\ref{subsec:notation}
gives some notation and basic definitions.
Section~\ref{sec:background} presents some necessary background
material for describing the ICG methods. Section~\ref{sec:accelerated_icg}
is split into three subsections. The first one precisely describes
the problem of interest, while the last two present the IA-ICG and
DA-ICG methods. Section~\ref{sec:spectral_details} describes an
efficient way of solving problem \eqref{eq:icg_mat_prox} by modifying
a solution of problem \eqref{eq:icg_vec_prox}. Section~\ref{sec:computational}
presents some numerical results. Section~\ref{sec:icg_cvg_rate}
establishes the iteration complexity of the ICG methods. Finally,
some auxiliary results are presented in Appendices~\ref{app:subdiff}
to \ref{app:spectral}.

\subsection{Notation and Basic Definitions}

\label{subsec:notation}

This subsection provides some basic notation and definitions.

The set of real numbers is denoted by $\r$. The set of non-negative
real numbers and the set of positive real numbers is denoted by $\r_{+}$
and $\r_{++}$ respectively. The set of natural numbers is denoted
by $\n$. The set of complex numbers is $\c$. The set of unitary
matrices of size $n$--by--$n$ is ${\cal U}^{n}$. For $t>0$,
define $\log_{1}^{+}(t):=\max\{1,\log(t)\}$. Let $\rn$ denote a
real--valued $n$--dimensional Euclidean space with norm $\|\cdot\|$.
Given a linear operator $A:\rn\mapsto\r^{p}$, the operator norm of
$A$ is denoted by $\|A\|:=\sup\{\|Az\|/\|z\|:z\in\rn,z\neq0\}$.
Using the asymptotic notation ${\cal O}$, we denote ${\cal O}_{1}(\cdot)\equiv{\cal O}(1+\cdot)$.

Let $(m,n)\in\n^{2}$ and let $r=\min\{m,n\}$. Given matrices $X\in\r^{m\times n}$
and $Y\in\r^{n\times n}$, let the quantities $\sigma(X)$ and $\lam(Y)$
denote the singular values and eigenvalues of $X$ and $Y$, respectively,
in nonincreasing order. Let ${\rm dg}:\r^{r}\mapsto\r^{r\times r}$
and ${\rm Dg}:\r^{m\times n}\mapsto\r^{r}$ be given pointwise by
\[
\left[\dg z\right]_{ij}=\begin{cases}
z_{i}, & \text{if }i=j,\\
0, & \text{otherwise},
\end{cases}\quad\left[\Dg Z\right]_{i}=Z_{ii},
\]
for every $z\in\r^{r},Z\in\r^{m\times n},$ and $(i,j)\in\{1,...,r\}^{2}$.

The following notation and definitions are for a general complete
inner product space ${\cal Z}$, whose inner product and its associated
induced norm are denoted by $\left\langle \cdot,\cdot\right\rangle $
and $\|\cdot\|$ respectively. Let $\psi:{\cal Z}\mapsto(-\infty,\infty]$
be given. The effective domain of $\psi$ is denoted by $\dom\psi:=\{x\in{\cal {\cal Z}}:\psi(x)<\infty\}$
and $\psi$ is said to be proper if $\dom\psi\neq\emptyset$. For
$\varepsilon\geq0$, the $\varepsilon$-subdifferential\emph{ }of
$\psi$ at $x\in\dom\psi$ is denoted by 
\[
\pt_{\varepsilon}\psi(z):=\left\{ w\in\rn:\psi(z')\geq\psi(z)+\left\langle w,z'-z\right\rangle -\varepsilon,\forall z'\in{\cal {\cal Z}}\right\} ,
\]
and we denote $\pt\psi\equiv\pt_{0}\psi$. The set of proper, lower
semi-continuous, convex functions is denoted by $\cConv{\cal Z}$.
The convex conjugate $\psi$ is denoted by $\psi^{*}$. The linear
approximation of $\psi$ at a point $z_{0}\in\dom\psi$ is denoted
by $\ell_{\psi}(\cdot;z_{0}):=\psi(z_{0})+\left\langle \nabla\psi(z_{0}),\cdot-z_{0}\right\rangle $.
The indicator of a closed convex set $C\subseteq{\cal Z}$ at a point
$z\in{\cal Z}$ is denoted by $\delta_{C}(z)$, which is $1$ if $z\in C$
and $\infty$ otherwise. The local Lipschitz constant of $\nabla\psi$
at two points $u,z\in{\cal Z}$ is denoted by 
\begin{equation}
L_{\psi}(x,y)=\begin{cases}
\frac{\|\nabla\psi(x)-\nabla\psi(y)\|}{\|x-y\|}, & x\neq y,\\
0, & x=y,
\end{cases}\quad\forall x,y\in\dom\psi.\label{eq:local_Lipsh}
\end{equation}

\section{Background Material}

\label{sec:background}

Recall from Section~\ref{sec:intro} that our interest is in solving
\eqref{eq:intro_prb} by repeatedly solving a sequence of prox subproblems
as in \eqref{eq:icg_mat_prox}. This section presents some background
material regarding \eqref{eq:icg_mat_prox}. 

This section considers the nonconvex composite optimization (NCO)
problem 
\begin{equation}
\min_{u\in{\cal Z}}\left\{ \psi(u):=\psi_{s}(u)+\psi_{n}(u)\right\} ,\label{eq:acg_motivating_prb}
\end{equation}
where ${\cal Z}$ is a finite dimensional inner product space and
the functions $\psi_{s}$ and $\psi_{n}$ are assumed to satisfy the
following assumptions: 
\begin{itemize}
\item[(B1)] $\psi_{n}\in\overline{{\rm Conv}}\ {\cal Z}$; 
\item[(B2)] $\psi_{s}$ is continuously differentiable on ${\cal Z}$ and satisfies
$\psi_{s}(u)-\ell_{\psi_{s}}(u;\icgY{})\leq M\|u-\icgY{}\|^{2}/2$
for some $M\geq0$ and every $u,\icgY{}\in{\cal Z}$. 
\end{itemize}
Clearly, problems \eqref{eq:intro_prb} and \eqref{eq:icg_mat_prox}
are special cases of \eqref{eq:acg_motivating_prb}, and hence any
definition or result that is stated in the context of \eqref{eq:acg_motivating_prb}
applies to \eqref{eq:intro_prb} and/or \eqref{eq:icg_mat_prox}.

An important notion of an approximate solution of \eqref{eq:acg_motivating_prb}
is as follows: given $\hat{\rho}>0$, a pair $(\icgY r,v_{r})$ is
said to be a $\hat{\rho}$--approximate solution of \eqref{eq:acg_motivating_prb}
if 
\begin{equation}
v_{r}\in\nabla\psi_{s}(\icgY r)+\pt\psi_{n}(\icgY r),\quad\|v_{r}\|\leq\hat{\rho}.\label{eq:gen_rho_approx_solution}
\end{equation}
In Section~\ref{sec:accelerated_icg}, we develop prox-type methods
for finding $\hat{\rho}$--approximate solutions of \eqref{eq:intro_prb}
that repeatedly solve \eqref{eq:icg_mat_prox} inexactly by taking
advantage of its spectral decomposition. 

We now discuss the inexactness criterion under which the subproblems
\eqref{eq:icg_mat_prox} are solved. Again, the criterion is described
in the context of \eqref{eq:acg_motivating_prb} as follows.

\vspace*{1em}
 \begin{mdframed}[nobreak=true] \textbf{Problem} ${\cal {A}}:$
Given $(\mu,\theta)\in\r_{++}^{2}$ and $\acgX 0\in\mathcal{Z}$,
find $(\icgY{},v,\varepsilon)\in\dom\psi\times\mathcal{Z}\times\mathbb{R}_{+}$
such that 
\begin{equation}
v\in\partial_{\varepsilon}\left(\psi-\frac{\mu}{2}\|\cdot-\icgY{}\|^{2}\right)(\icgY{}),\quad\|v\|^{2}+2\varepsilon\le\theta^{2}\|\icgY{}-\acgX 0\|^{2}.\label{eq:cvx_inexact}
\end{equation}
\end{mdframed} \vspace*{1em}

We begin by making three remarks about the above problem. First, if
$(\icgY{},v,\varepsilon)$ solves Problem~${\cal {A}}$ with $\theta=0$,
then $(v,\varepsilon)=(0,0)$, and $z$ is an exact solution of \eqref{eq:acg_motivating_prb}.
Hence, the output $(\icgY{},v,\varepsilon)$ of Problem~${\cal {A}}$
can be viewed as an inexact solution of \eqref{eq:acg_motivating_prb}
when $\theta\in\r_{++}$. Second, the input $\acgX 0$ is arbitrary
for the purpose of this section. However, the two methods described
in Section~\ref{sec:accelerated_icg} for solving \eqref{eq:intro_prb}
repeatedly solve \eqref{eq:icg_mat_prox} according to Problem~${\cal A}$
with the input $\acgX 0$ at the $k^{{\rm th}}$ iteration determined
by the iterates generated at the $(k-1)^{{\rm th}}$ iteration. Third,
defining the function 
\begin{equation}
\Delta_{\mu}(u;\icgY{},v):=\psi(\icgY{})-\psi(u)-\inner v{\icgY{}-u}+\frac{\mu}{2}\|u-\icgY{}\|^{2}\quad\forall u\in\dom\psi, \label{eq:Delta_def}
\end{equation}
another way to express the inclusion in \eqref{eq:cvx_inexact} is
$\Delta_{\mu}(u;\icgY{},v)\leq\varepsilon$ for every $u\in\dom\psi$.
Finally, the relaxed ACG (R-ACG) algorithm presented later in this
subsection will be shown to solve Problem~${\cal A}$ when $\psi_{s}$
is convex. Moreover, it solves a weaker version of Problem~${\cal A}$
involving $\Delta_{\mu}$ (see Problem~${\cal B}$ later on) whenever
$\psi_{s}$ is not convex and as long as some key inequalities are
satisfied during its execution.

A technical issue in our analysis in this paper lies in the ability
of refining the output of Problem~${\cal A}$ to an approximate solution
$(\icgY r,v_{r})$ of \eqref{eq:acg_motivating_prb}, i.e., one satisfying
the inclusion in \eqref{eq:gen_rho_approx_solution}, in which $\|v_{r}\|$
is nicely bounded. We now present a refinement procedure that addresses
this issue.

\noindent %
\noindent\begin{minipage}[t]{1\columnwidth}%
\noindent \rule[0.5ex]{1\columnwidth}{1pt}

\noindent \textbf{Refinement Procedure}

\noindent \rule[0.5ex]{1\columnwidth}{1pt}%
\end{minipage}

\noindent \textbf{Input}: a triple $(M,\psi_{s},\psi_{n})$ satisfying
(B1)--(B2) and a pair $(\icgY{},v)\in\dom\psi_{n}\times{\cal Z}$;

\noindent \textbf{Output}: a pair $(\icgY r,v_{r})$ satisfying the
inclusion in \eqref{eq:gen_rho_approx_solution}; 
\begin{itemize}
\item[1.] set the quantities 
\begin{align}
\icgY r & =\argmin_{u\in{\cal Z}}\left\{ \left\langle \nabla\psi_{s}(\icgY{})-v,u\right\rangle +\frac{M}{2}\|u-\icgY{}\|^{2}+\psi_{n}(u)\right\} ,\label{eq:z_hat_def}\\
v_{r} & =v+M(\icgY{}-\icgY r)+\nabla\psi_{s}(\icgY r)-\nabla\psi_{s}(\icgY{}),\label{eq:v_hat_def}
\end{align}
and output $(\icgY r,v_{r})$. 
\end{itemize}
\noindent \rule[0.5ex]{1\columnwidth}{1pt}

The result below presents the key properties of the above procedure.
For the sake of brevity, we write $(\icgY r,v_{r})=RP(\icgY{},v)$
to indicate that the pair $(\icgY r,v_{r})$ is the output of the
above procedure with inputs $(M,\psi_{s},\psi_{n})$ and $(\icgY{},v)$. 
\begin{prop}
\label{prop:gen_refinement}Let $(M,\psi_{s},\psi_{n})$ satisfying
assumptions (B1)--(B2) and a triple $(\icgY{},v,\varepsilon)\in\dom\psi_{n}\times{\cal Z}\times\r_{+}$
be given. Moreover, let $(\icgY r,v_{r})=RP(\icgY{},v)$, denote $L_{\psi_{s}}(\cdot,\cdot)$
simply by $L(\cdot,\cdot)$ where $L_{\psi_{s}}(\cdot,\cdot)$ is
as in \eqref{eq:local_Lipsh}, and let $\Delta_{\mu}$ be as in \eqref{eq:Delta_def}.
Then, the following statements hold: 
\begin{itemize}
\item[(a)] $v_{r}\in\nabla\psi_{s}(\icgY r)+\pt\psi_{n}(\icgY r)$; 
\item[(b)]  $\Delta_{\mu}(\icgY r;\icgY{},v)\geq {M}\|\icgY r-\icgY{}\|^{2} / 2$;
\item[(c)] if $\Delta_{\mu}(\icgY r;\icgY{},v)\le\varepsilon$ and $(\icgY{},v,\varepsilon)$
satisfies the inequality in \eqref{eq:cvx_inexact}, then 
\begin{equation}
\|v_{r}\|\leq\theta\left[1+\frac{M+L(\icgY{},\icgY r)}{\sqrt{M}}\right]\|\icgY{}-\acgX 0\|;\label{eq:v_r_acg_bd}
\end{equation}
\item[(d)] if $(\icgY{},v,\varepsilon)$ solves Problem ${\cal A}$, then $\Delta_{\mu}(u;\icgY{},v)\le\varepsilon$
for every $u\in\dom\psi_{n}$, and, as a consequence, bound \eqref{eq:v_r_acg_bd}
holds.
\end{itemize}
\end{prop}


\begin{proof}
(a) Using the definition of $v_{r}$ and the optimality of $\icgY r$,
we have that 
\begin{align*}
v_{r}=v+M(\icgY{}-\icgY r)+\nabla\psi_{s}(\icgY r)-\nabla\psi_{s}(\icgY{})\in\nabla\psi_{s}(\icgY r)+\pt\psi_{n}(\icgY r).
\end{align*}

(b)  Let $(\icgY{},v)\in\dom\psi_{n}\times{\cal Z}$ be fixed, and
define $\widetilde{\psi}_{s}:=\psi_{s}-\langle v,\cdot\rangle$. Using
Proposition~\ref{prop:basic_refinement} with $(g,h,L)=(\widetilde{\psi}_{s},\psi_{n},M)$
and $(z,\hat{z})=(\icgY{},\icgY r)$, and the definition of $\Delta_\mu$ in \eqref{eq:Delta_def}, we have 
\begin{align*}
\frac{M}{2}\|\icgY{}-\icgY r\|^{2} & \leq(\widetilde{\psi}_{s}+\psi_{n})(\icgY{})-(\widetilde{\psi}_{s}+\psi_{n})(\icgY r)\\
 & = \psi(\icgY{})-\psi(\icgY r)-\langle v,\icgY{}-\icgY r\rangle \leq \Delta_{\mu}(\icgY r;\icgY{},v).
\end{align*}

(c) Using the assumption that $\Delta_{\mu}(\icgY r;\icgY{},v)\leq\varepsilon$,
part (b), and the inequality in \eqref{eq:cvx_inexact}, we have that
\begin{equation}
\|\icgY{}-\icgY r\|\leq\sqrt{\frac{2\Delta_{\mu}(\icgY r;\icgY{},v)}{M}}\leq\sqrt{\frac{2\varepsilon}{M}}\leq\frac{\theta}{\sqrt{M}}\|\icgY{}-\acgX 0\|.\label{eq:z_zr_bd}
\end{equation}
Using the triangle inequality, the definition of $L(\cdot,\cdot)$,
\eqref{eq:z_zr_bd} and the inequality in \eqref{eq:cvx_inexact}
again, we conclude that 
\begin{align*}
\|v_{r}\| & \leq\|v\|+\left[M+L(\icgY{},\icgY r)\right]\cdot\|\icgY{}-\icgY r\|\leq\theta\left[1+\frac{M+L(\icgY{},\icgY r)}{\sqrt{M}}\right]\|\icgY{}-\acgX 0\|.
\end{align*}
(d) The fact that $\Delta_{\mu}(u;\icgY{},v)\leq\varepsilon$ for
every $u\in\dom\psi_{n}$ follows immediately from the inclusion in
\eqref{eq:cvx_inexact} and the definition of $\Delta_{\mu}$ in \eqref{eq:Delta_def}.
The fact that \eqref{eq:v_r_acg_bd} holds now follows from part (c). 
\end{proof}
We make a few remarks about Proposition~\ref{prop:gen_refinement}.
First, it follows from (a) that $(\icgY r,v_{r})$ satisfies the inclusion
in \eqref{eq:gen_rho_approx_solution}. Second, it follows from (a)
and (c) that if $\theta=0$, then $(\varepsilon,v_{r})=(0,0)$, and hence
$\icgY r$ is an exact stationary point of \eqref{eq:acg_motivating_prb}.
In general, \eqref{eq:v_r_acg_bd} implies that the residual $\|v_{r}\|$
is directly proportional to $\|\icgY{}-w\|$, and hence, becomes smaller
as this quantity approaches \textit{zero}.

Inequality \eqref{eq:v_r_acg_bd} plays an important technical role
in the complexity analysis of the two prox-type methods of Section~\ref{sec:accelerated_icg}.
Sufficient conditions for its validity are provided in (c) and (d),
with (c) being the weaker one, in view of (d). When $\psi_{s}$ is
convex, it is shown that every iterate of the R-ACG algorithm presented
below always satisfies the inclusion in \eqref{eq:cvx_inexact}, and
hence, verifying the the validity of the sufficient condition in (c)
amounts to simply checking whether the inequality in \eqref{eq:cvx_inexact}
holds. When $\psi_{s}$ is not convex, verification of the inclusion
in \eqref{eq:cvx_inexact}, and hence the sufficient condition in
(d), is generally not possible, while the one in (c) is. This is a
major advantage of the sufficient condition in (c), which is exploited
in this paper towards the development of adaptive prox-type methods
which attempt to approximately solve \eqref{eq:acg_motivating_prb}
when $\psi_{s}$ is not convex.

For the sake of future reference, we now state the following problem
for finding a triple $(\icgY{},v,\varepsilon)$ satisfying the sufficient
condition in Proposition~\ref{prop:gen_refinement}(c). Its statement
relies on the refinement procedure preceding Proposition~\ref{prop:gen_refinement}.

\vspace*{1em}
 \begin{mdframed}[nobreak=true] \textbf{Problem} ${\cal {B}}:$
Given the same inputs as in Problem~${\cal A}$, find $(\icgY{},v,\varepsilon)\in\dom\psi\times\mathcal{Z}\times\mathbb{R}_{+}$
satisfying the inequality in \eqref{eq:cvx_inexact} and 
\begin{equation}
\Delta_{\mu}(\icgY r;\icgY{},v)\leq\varepsilon,\label{eq:prb_B_Delta_ineq}
\end{equation}
where $\Delta_{\mu}(\cdot;\cdot,\cdot)$ is as in \eqref{eq:Delta_def}
and $\icgY r$ is the first component of the refined pair $(\icgY r,v_{r})=RP(\icgY{},v)$.
\end{mdframed} \vspace*{1em}

We now state the aforementioned R-ACG algorithm which solves Problem~${\cal A}$
when $\psi_{s}$ is convex and solves Problem~${\cal B}$ whenever
$\psi_{s}$ is not convex and two key inequalities are satisfied,
one at every iteration (i.e., \eqref{eq:acg_ineq_invar}) and one
at the end of its execution.

\noindent %
\noindent\begin{minipage}[t]{1\columnwidth}%
\noindent \rule[0.5ex]{1\columnwidth}{1pt}

\noindent \textbf{R-ACG Algorithm}

\noindent \rule[0.5ex]{1\columnwidth}{1pt}%
\end{minipage}

\noindent \textbf{Input}: a quadruple $(\mu,M,\psi_{s},\psi_{n})$
satisfying (B1)--(B2) and a pair $(\theta,\acgX 0)$;

\noindent \textbf{Output}: a triple $(\icgY{},v,\varepsilon)$ that
solves Problem~${\cal B}$ or a \textit{failure} status; 
\begin{itemize}
\item[0.] define $\psi:=\psi_{s}+\psi_{n}$ and set $\acgY 0=\acgX 0$, $B_{0}=0$,
$\Gamma_{0}\equiv0$, and $j=1$; 
\item[1.] compute the iterates 
\begin{align*}
\xi_{j-1} & =\frac{1+\mu B_{j-1}}{M-\mu},\quad b_{j-1}=\frac{\xi_{j-1}+\sqrt{\xi_{j-1}^{2}+4\xi_{j-1}B_{j-1}}}{2},\\
B_{j} & =B_{j-1}+b_{j-1},\quad\acgXTilde{j-1}=\frac{B_{j-1}}{B_{j}}\acgX{j-1}+\frac{b_{j-1}}{B_{j}}\acgY{j-1},\\
\acgX j & =\argmin_{u\in{\cal Z}}\left\{ l_{\psi_{s}}(u;\acgXTilde{j-1})+\psi_{n}(u)+\frac{M}{2}\|u-\acgXTilde{j-1}\|^{2}\right\} ,\\
\acgY j & =\frac{1}{1+\mu B_{j}}\left[\acgY{j-1}-b_{j-1}(M-\mu)(\acgXTilde{j-1}-\acgX j)+\mu(B_{j-1}\acgY{j-1}+b_{j-1}\acgX j)\right];
\end{align*}
\item[2.] compute the quantities 
\begin{align*}
\tilde{\gamma}_{j} & =l_{\psi_{s}}(\cdot;\acgXTilde{j-1})+\psi_{n}+\frac{\mu}{2}\|\cdot-\acgXTilde{j-1}\|^{2},\\
\gamma_{j} & =\tilde{\gamma}_{j}(\acgX j)+(M-\mu)\left\langle \acgXTilde{j-1}-\acgX j,\cdot-\acgX j\right\rangle +\frac{\mu}{2}\|\cdot-\acgX j\|^{2},\\
\Gamma_{j} & =\frac{B_{j-1}}{B_{j}}\Gamma_{j-1}+\frac{b_{j-1}}{B_{j}}\gamma_{j},\quad\acgU j=\frac{\acgY 0-\acgY j}{B_{j}}+\mu(\acgY j-\acgX j),\\
\eta_{j} & = \max\left\{0,\psi(\acgX j)-\Gamma_{j}(\acgY j)-\left\langle \acgU j,\acgX j-\acgY j\right\rangle +\frac{\mu}{2}\|\acgX j-\acgY j\|^{2}\right\};
\end{align*}
\item[3.] if the inequality
\begin{gather}
\left(\frac{1}{1+\mu B_{j}}\right)\|B_{j}\acgU j+\acgX j-\acgX 0\|^{2}+2B_{j}\eta_{j}\leq\|\acgX j-\acgX 0\|^{2}\label{eq:acg_ineq_invar}
\end{gather}
holds, then go to step~4; otherwise, \textbf{stop} with a\emph{ failure}
status; 
\item[4.] if the inequality 
\begin{align}
\|\acgU j\|^{2}+2\eta_{j} & \leq\theta^{2}\|\acgX j-\acgX 0\|^{2},\label{eq:acg_ineq}
\end{align}
holds, then go to step~5; otherwise, go to step 1; 
\item[5.] set $(\icgY{},v,\varepsilon)=(\acgX j,\acgU j,\eta_{j})$ and compute
$(\icgY r,v_{r})=RP(\acgX j,\acgU j)$; if the condition 
\[
\Delta_{\mu}(\icgY r;\icgY{},v)\leq\varepsilon,
\]
holds then \textbf{stop} with a \emph{success }status and \textbf{output}
the triple $(\icgY{},v,\varepsilon)$; otherwise, \textbf{stop} with
a \emph{failure} status.
\end{itemize}
\noindent \rule[0.5ex]{1\columnwidth}{1pt}

It is well-known (see, for example, \cite[Proposition 2.3]{he2016accelerated}) that the scalar $B_{j}$ updated in step~1 satisfies
\begin{equation}
B_{j}\geq\frac{1}{M}\max\left\{ \frac{j^{2}}{4},\left(1+\sqrt{\frac{\mu}{4M}}\right)^{2(j-1)}\right\} \quad\forall j\geq1.\label{eq:B_j_bd}
\end{equation}
The next result presents the key properties about the R-ACG algorithm. 
\begin{prop}
\label{prop:acg_properties}The R-ACG algorithm has the following
properties: 
\begin{itemize}
\item[(a)] it stops with either failure or success in 
\begin{equation}
{\cal O}\left(\left[1+\sqrt{\frac{L}{\mu}}\right]\log_{1}^{+}\left[LK_{\theta}(1+\mu K_{\theta})\right]\right)\label{eq:r_acg_total_compl}
\end{equation}
iterations, where $K_{\theta}:=1+\sqrt{2}/\theta$; 
\item[(b)] if it stops with success, then its output $(\icgY{},v,\varepsilon)$
solves Problem~${\cal B}$; 
\item[(c)] if $\psi_{s}$ is $\mu$--strongly convex then it always stops with
success and its output $(\icgY{},v,\varepsilon)$ solves Problem~${\cal A}$. 
\end{itemize}
\end{prop}

\begin{proof}
(a) See Appendix~\ref{app:r_acg}.

(b) This follows from the successful checks in step~4 and 5 of the
algorithm.

(c) The fact that the algorithm never stops with failure follows from
Proposition~\ref{prop:acg_key_props}(c)--(d) in Appendix~\ref{app:r_acg}.
The fact that the algorithm stops with success follows from the previous
statement, the successful checks in step~4 and 5 of the algorithm,
and the fact that the algorithm stops in a finite number of iterations
in part (a). 
\end{proof}

\section{\label{sec:accelerated_icg}Inexact Composite Gradient Methods}

This section presents the ICG methods and the general problem that
they solve. It contains three subsections. The first one presents
the problem of interest and gives a general outline of the ICG methods,
the second one presents the IA-ICG method, and the third one presents
the DA-ICG method. For the ease of presentation, the proofs in this section are deferred
to Section~\ref{sec:icg_cvg_rate}.

\subsection{\label{subsec:prb_of_interest}Problem of Interest and Outline of
the Methods}

This subsection describes the problem that the ICG methods solve and outlines
their structure.

Instead of considering problems having
the spectral structure
mentioned in Section~\ref{sec:intro},
this section considers
a more general NCO problem where its variable $u$
lies in a finite dimensional inner product space ${\cal Z}$
(and, hence, can be either a vector and/or matrix) and presents 
both ICG methods in this
more general setting.
Section~\ref{sec:spectral_details} then presents a modification
of the ACG subroutine used by both ICG methods that
drastically improves their efficiency in the setting of 
the spectral problem \eqref{eq:intro_prb}.

More specifically, this section considers the problem
\begin{equation}
\min_{u\in{\cal Z}}\left[\phi(u):=f_{1}(u)+f_{2}(u)+h(u)\right]\label{eq:main_prb}
\end{equation}
where the functions $f_{1},f_{2},$ and $h$ are assumed to satisfy the
following assumptions: 
\begin{itemize}
\item[(A1)] $h\in\cConv{\cal Z}$; 
\item[(A2)] $f_{1},f_{2}$ are continuously differentiable functions and there
exists $(m_{1},M_{1})\in\r^{2}$ and $(m_{2},M_{2})\in\r^{2}$ such
that, for $i\in\{1,2\}$, we have 
\begin{gather}
-\frac{m_{i}}{2}\|u-\icgY{}\|^{2}\leq f_{i}(u)-\ell_{f_{i}}(u;\icgY{})\leq\frac{M_{i}}{2}\|u-\icgY{}\|^{2}\quad\forall u,\icgY{}\in\dom h;\label{eq:curv_fi}
\end{gather}
\item[(A3)] for $i\in\{1,2\}$, we have 
\[
\|\nabla f_{i}(u)-\nabla f_{i}(\icgY{})\|\leq L_{i}\|u-\icgY{}\|\quad\forall u,\icgY{}\in\dom h,
\]
where $L_{i}:=\max\{|m_{i}|,|M_{i}|\}$; 
\item[(A4)] $\phi_{*}:=\inf_{u\in{\cal Z}}\phi(u)>-\infty$. 
\end{itemize}
Note that assumption (A2) implies that assumption (A3) holds when
the interior of $\dom h$ is nonempty. Under the above assumptions,
the proposed ICG methods find an approximate solution $(\hat{\icgY{}},\hat{v})$
of \eqref{eq:main_prb} as in \eqref{eq:gen_rho_approx_solution}
with $\psi_{s}=f_{1}+f_{2}$ and $\psi_{n}=h$, i.e. 
\begin{equation}
\hat{v}\in\nabla f_{1}(\hat{\icgY{}})+\nabla f_{2}(\hat{\icgY{}})+\pt h(\hat{\icgY{}}),\quad\|\hat{v}\|\leq\hat{\rho}.\label{eq:rho_approx_soln}
\end{equation}

We now outline the ICG methods. Given a starting point $\icgY 0\in\dom\psi_{n}$
and a special stepsize $\lam>0$, each method continually calls the
R-ACG algorithm of Section~\ref{sec:background} to find an approximate
solution of a prox-linear form of \eqref{eq:main_prb}. More specifically,
each R-ACG call is used to tentatively find an approximate solution
of 
\begin{equation}
\min_{u\in{\cal Z}}\left[\psi(u)=\lam\left[\ell_{f_{1}}(u;\acgX 0)+f_{2}(u)+h(u)\right]+\frac{1}{2}\|u-\acgX 0\|^{2}\right],\label{eq:icg_subprb}
\end{equation}
for some reference point $\acgX 0$. For the IA-ICG method, the point
$\acgX 0$ is $\icgY 0$ for the first R-ACG call and is the last
obtained approximate solution for the other R-ACG calls. For the DA-ICG
method, the point $\acgX 0$ is chosen in an accelerated manner.

From the output of the $k^{{\rm th}}$ R-ACG call, a refined pair
$(\hat{\icgY{}},\hat{v})=(\hat{\icgY{}}_{k},\hat{v}_{k})$ is generated
which: (i) always satisfies the inclusion of \eqref{eq:rho_approx_soln};
and (ii) is such that $\min_{i\leq k}\|\hat{v}_{i}\|\to0$ as $k\to\infty$.
More specifically, this refined pair is generated by applying the
refinement procedure of Section~\ref{sec:background} and adding
some adjustments to the resulting output to conform with our goal
of finding an approximate solution as in \eqref{eq:rho_approx_soln}.
For the ease of future reference, we now state this specialized refinement
procedure. Before proceeding, we introduce the shorthand notation
\begin{equation}
M_{i}^{+}:=\max\left\{ M_{i},0\right\} ,\quad m_{i}^{+}:=\max\left\{ m_{i},0\right\} ,\quad L_{i}(x,y):=L_{f_{i}}(x,y),\label{eq:mMi_def}
\end{equation}
for $i\in\{1,2\}$, to keep its presentation (and future results)
concise.

\noindent %
\noindent\begin{minipage}[t]{1\columnwidth}%
\noindent \rule[0.5ex]{1\columnwidth}{1pt}

\noindent \textbf{Specialized Refinement Procedure}

\noindent \rule[0.5ex]{1\columnwidth}{1pt}%
\end{minipage}

\noindent \textbf{Input}: a quadruple $(M_{2},f_{1},f_{2},h)$ satisfying
(A1)--(A2), a scalar $\lam>0$, and a triple $(\icgY{},v,\acgX 0)\in\dom\psi_{n}\times{\cal Z}\times{\cal Z}$;

\noindent \textbf{Output}: a pair $(\hat{\icgY{}},\hat{v})$ satisfying
the inclusion of \eqref{eq:rho_approx_soln}; 
\begin{itemize}
\item[1.] compute $(\hat{\icgY{}},v_{r})=RP(\icgY{},v)$ using the refinement
procedure in Section~\ref{sec:background} with 
\begin{equation}
M=\lam M_{2}^{+}+1,\quad\psi_{s}=\lam\left[\ell_{f_{1}}(\cdot;\acgX 0)+f_{2}\right]+\frac{1}{2}\|\cdot-\acgX 0\|^{2},\quad\psi_{n}=\lam h;\label{eq:srp_subinputs}
\end{equation}
\item[2.] compute the residual 
\[
\hat{v}=\frac{1}{\lam}(v_{r}+\acgX 0-\icgY{})+\nabla f_{1}(\hat{\icgY{}})-\nabla f_{1}(\acgX 0),
\]
and output $(\hat{\icgY{}},\hat{v})$. 
\end{itemize}
\noindent \rule[0.5ex]{1\columnwidth}{1pt}

The result below states some properties about the above procedure.
For the sake of brevity, we write $(\hat{\icgY{}},\hat{v})=SRP(\icgY{},v,\acgX 0)$
to indicate that the pair $(\hat{\icgY{}},\hat{v})$ is the output
of the above procedure with inputs $(M_{2},f_{1},f_{2},h)$, $\lam$,
and $(\icgY{},v,\acgX 0)$.
\begin{lem}
\label{lem:spec_refine} Let $(m_{1},M_{1})$, $(m_{2},M_{2})$, and
$(f_{1},f_{2},h)$ satisfying assumptions (A1)--(A3) and a quadruple
$(\acgX 0,\icgY{},v,\varepsilon)\in{\cal Z}\times\dom\psi_{n}\times{\cal Z}\times\r_{+}$
be given. Moreover, let $(\hat{\icgY{}},\hat{v})=SRP(\icgY{},v,\acgX 0)$
and define 
\begin{gather}
C_{\lam}(x,y):=\frac{1+\lam\left[M_{2}^{+}+L_{1}(x,y)+L_{2}(x,y)\right]}{\sqrt{1+\lam M_{2}^{+}}},\label{eq:C_lam_fn_def}
\end{gather}
for every $x,y\in{\cal Z}$. Then, the following statements hold:
\begin{itemize}
\item[(a)] $\hat{v}\in\nabla f_{1}(\hat{\icgY{}})+\nabla f_{2}(\hat{\icgY{}})+\pt h(\hat{\icgY{}})$; 
\item[(b)] if $(\icgY{},v,\varepsilon)$ solves Problem~${\cal B}$ with $(\mu,\psi_{s},\psi_{n})$
as in \eqref{eq:gen_acg_inputs}, then 
\[
\|\hat{v}\|\leq\left[L_{1}(\icgY{},w)+\frac{2+\theta C_{\lam}(\icgY{},\hat{\icgY{}})}{\lam}\right]\|\icgY{}-\acgX 0\|.
\]
\end{itemize}
\end{lem}

It is worth recalling from Section~\ref{sec:intro} that in the applications
we consider, the cost of the R-ACG call is small compared to SVD computation
that is performed before solving each subproblem as in \eqref{eq:icg_subprb}.
Hence, in the analysis that follows, we present complexity results
related to the number of subproblems solved rather than the total
number of R-ACG iterations. We do note, however, that the number of
R-ACG iterations per subproblem is finite in view of Proposition~\ref{prop:acg_properties}(a).

\subsection{Static and Dynamic IA-ICG Methods}

This subsection presents the static and dynamic IA-ICG methods.

We first state the static IA-ICG method.

\noindent %
\noindent\begin{minipage}[t]{1\columnwidth}%
\noindent \rule[0.5ex]{1\columnwidth}{1pt}

\noindent \textbf{Static IA-ICG Method}

\noindent \rule[0.5ex]{1\columnwidth}{1pt}%
\end{minipage}

\noindent \textbf{Input}: function triple $(f_{1},f_{2},h)$ and scalar
quadruple $(m_{1},M_{1},m_{2},M_{2})\in\r^{4}$ satisfying (A1)--(A4),
tolerance $\hat{\rho}>0$, initial point $\icgY 0\in\dom h$, and
scalar pair $(\lam,\theta)\in\r_{++}\times(0,1)$ satisfying 
\begin{equation}
\lam M_{1}+\theta^{2}\leq\frac{1}{2};\label{eq:aicg_lam_restr}
\end{equation}

\noindent \textbf{Output}: a pair $(\hat{\icgY{}},\hat{v})$ satisfying
\eqref{eq:rho_approx_soln} or a \textit{failure} status; 
\begin{itemize}
\item[0.] let $\Delta_{1}(\cdot;\cdot,\cdot)$ be as in \eqref{eq:Delta_def}
with $\mu=1$, and set $k=1$;
\item[1.] use the R-ACG algorithm to tentatively solve Problem~${\cal B}$
associated with \eqref{eq:icg_subprb}, i.e., with inputs $(\mu,M,\psi_{s},\psi_{n})$
and $(\theta,\acgX 0)$ where the former is given by 
\begin{equation}
\begin{gathered}\mu=1,\quad M=\lam M_{2}^{+}+1,\\
\psi_{s}=\lam\left[\ell_{f_{1}}(\cdot;\acgX 0)+f_{2}\right]+\frac{1}{2}\|\cdot-\acgX 0\|^{2},\quad\psi_{n}=\lam h,
\end{gathered}
\label{eq:gen_acg_inputs}
\end{equation}
and $\acgX 0=\icgY{k-1}$; if the R-ACG stops with \emph{failure},
then \textbf{stop} with a \textit{failure} status; otherwise, let
$(\icgY k,v_{k},\varepsilon_{k})$ denote its output and go to step~2; 
\item[2.] if the inequality $\Delta_{1}(\icgY{k-1};\icgY k,v_{k})\leq\varepsilon_{k}$
holds, then go to step~3; otherwise, \textbf{stop} with a \textit{failure}
status; 
\item[3.] set $(\hat{\icgY{}}_{k},\hat{v}_{k})=SRP(\icgY k,v_{k},\icgY{k-1})$;
if $\|\hat{v}_{k}\|\leq\hat{\rho}$ then \textbf{stop} with a \textit{success}
status and \textbf{output} $(\hat{\icgY{}},\hat{v})=(\hat{\icgY{}}_{k},\hat{v}_{k})$;
otherwise, update $k\gets k+1$ and go to step~1. 
\end{itemize}
\noindent \rule[0.5ex]{1\columnwidth}{1pt}

Note that the static IA-ICG method may fail without obtaining a pair
satisfying \eqref{eq:rho_approx_soln}. In Theorem~\ref{thm:aicg_compl}(c)
below, we state that a sufficient condition for the method to stop
successfully is that $f_{2}$ be convex. This property will be important
when we present the dynamic IA-ICG method, which: (i) repeatedly
calls the static method; and (ii) incrementally transfers convexity
from $f_{1}$ to $f_{2}$ between each call until a successful termination
is achieved.

We now make some additional remarks about the above method. First,
it performs two kinds of iterations, namely, ones that are indexed
by $k$ and ones that are performed by the R-ACG algorithm. We refer
to the former kind as outer iterations and the latter kind as inner
iterations. Second, in view of \eqref{eq:aicg_lam_restr}, if $M_{1}>0$
then $0<\lam<(1-2\theta^{2})/(2M_{1})$ whereas if $M_{1}\leq0$ then
$0<\lam<\infty$.  Finally, the most expensive part of the method is the R-ACG call in step~1. 
In Section~\ref{sec:spectral_details}, we show that this call
can be replaced with a call to a spectral version of R-ACG that
is dramatically more efficient when the underlying problem 
has the spectral structure as in \eqref{eq:intro_prb}.

The next result summarizes some facts about the static IA-ICG method.
Before proceeding, we first define some useful quantities. For $\lam>0$
and $u,w\in{\cal Z}$, define 
\begin{gather}
\widetilde{\ell}_{\phi}(u;w):=\ell_{f_{1}}(u;w)+f_{2}(u)+h(u),\quad\overline{C}_{\lam}:=\frac{1+\lam(M_{2}^{+}+L_{1}+L_{2})}{\sqrt{1+\lam M_{2}^{+}}}.\label{eq:ell_phi_C_bar_lam_def}
\end{gather}

\begin{thm}
\label{thm:aicg_compl}The following statements hold about the static
IA-ICG method: 
\begin{itemize}
\item[(a)] it stops in 
\begin{equation}
{\cal O}_{1}\left(\left[\sqrt{\lam}L_{1}+\frac{1+\theta\overline{C}_{\lam}}{\sqrt{\lam}}\right]^{2}\left[\frac{\phi(z_{0})-\phi_{*}}{\hat{\rho}^{2}}\right]\right)\label{eq:aicg_outer_compl}
\end{equation}
outer iterations, where $\phi_{*}$ is as in (A4); 
\item[(b)] if it stops with success, then its output pair $(\hat{\icgY{}},\hat{v})$
is a $\hat{\rho}$--approximate solution of \eqref{eq:main_prb}; 
\item[(c)] if $f_{2}$ is convex, then it always stops with success. 
\end{itemize}
\end{thm}

We now make three remarks about the above results. First, if $\theta={\cal O}(1/\overline{C}_{\lam})$
then \eqref{eq:aicg_outer_compl} is on the order of
\begin{equation}
{\cal O}_{1}\left(\left[\sqrt{\lam}L_{1}+\frac{1}{\sqrt{\lam}}\right]^{2}\left[\frac{\phi(z_{0})-\phi_{*}}{\hat{\rho}^{2}}\right]\right).\label{eq:compl_aicg}
\end{equation}
Moreover, comparing the above complexity to the iteration complexity
of the ECG method described in Section~\ref{sec:intro}, which is
known (see, for example, \cite{nesterov2012gradient})
to obtain an approximate solution of \eqref{eq:main_prb} in 
\begin{equation}
{\cal O}_{1}\left(\left[\sqrt{\lam}(L_{1}+L_{2})+\frac{1}{\sqrt{\lam}}\right]^{2}\left[\frac{\phi(z_{0})-\phi_{*}}{\hat{\rho}^{2}}\right]\right)\label{eq:compl_ecg}
\end{equation}
iterations, we see that \eqref{eq:compl_aicg} is smaller than \eqref{eq:compl_ecg}
in magnitude when $L_{2}$ is large. Notice also that the 
complexity in \eqref{eq:compl_aicg} corresponds to applying 
the ECG method to \eqref{eq:intro_prb} where the composite function is
$f_2+h$ instead of just $h$. Second, Theorem~\ref{thm:aicg_compl}(b)
shows that if the method stops with success, regardless of the convexity
of $f_{2}$, then its output pair $(\hat{\icgY{}},\hat{v})$ is always
an approximate solution of \eqref{eq:main_prb}. Third, in view of
Proposition~\ref{prop:aicg_v_hat_rate_alt}, the quantities $L_{1}$
and $\overline{C}_{\lam}$ in all of the previous complexity results
can be replaced by their averaged counterparts in \eqref{eq:avg_def}.
As these averaged quantities only depend on $\{(\icgY i,\hat{\icgY{}}_{i})\}_{i=1}^{k}$,
we can infer that the static IA-ICG method adapts to the local geometry
of its input functions.

We now state the dynamic IA-ICG method that resolves the
issue of failure in the static IA-ICG method.

\noindent %
\noindent\begin{minipage}[t]{1\columnwidth}%
\noindent \rule[0.5ex]{1\columnwidth}{1pt}

\noindent \textbf{Dynamic IA-ICG Method}

\noindent \rule[0.5ex]{1\columnwidth}{1pt}%
\end{minipage}

\noindent \textbf{Input}: the same as the static IA-ICG method but
with an additional parameter $\xi_{0}>0$;

\noindent \textbf{Output}: a pair $(\hat{\icgY{}},\hat{v})$ satisfying
\eqref{eq:rho_approx_soln}; 
\begin{itemize}
\item[0.] set $\xi=\xi_{0}$, $\ell=1$, and 
\begin{gather}
\begin{gathered}f_{1}=f_{1}-\frac{\xi}{2}\|\cdot\|^{2},\quad f_{2}=f_{2}+\frac{\xi}{2}\|\cdot\|^{2},\\
m_{1}=m_{1}+\xi,\quad M_{1}=M_{1}-\xi,\quad m_{2}=m_{2}-\xi,\quad M_{2}=M_{2}+\xi;
\end{gathered}
\label{eq:perturbed_quants}
\end{gather}
\item[1.]  call the static IA-ICG method with inputs $(f_{1},f_{2},h)$, $(m_{1},M_{1},m_{2},M_{2})$,
$\hat{\rho}$, $\icgY 0$, and $(\lam,\theta)$; 
\item[2.] if the static IA-ICG call stops with a \textit{failure} status, then
set $\xi=2\xi$, update the quantities in \eqref{eq:perturbed_quants}
with the new value of $\xi$, increment $\ell=\ell+1$, and go to
step~1; otherwise, let $(\hat{\icgY{}},\hat{v})$ be the output pair
returned by the static IA-ICG call, \textbf{stop}, and \textbf{output}
this pair. 
\end{itemize}
\rule[0.5ex]{1\columnwidth}{1pt}

Some remarks about the above method are in order. First, in view of
\eqref{eq:aicg_lam_restr} and the fact that $M_{1}$ is monotonically
decreasing, the parameter $\lam$ does not need to be changed for
each IA-ICG call. Second, in view of assumption (A2) and Theorem~\ref{thm:aicg_compl}(c),
the IA-ICG call in step~1 always terminates with success whenever
$m_{2}\leq0$. As a consequence, the total number of IA-ICG calls
is at most $\lceil \log(2m_{2}^{+}/\xi_{0}) \rceil $. Third,
in view of the second remark and Theorem~\ref{thm:aicg_compl}(b),
the method always obtains a $\hat{\rho}$--approximate solution
of \eqref{eq:main_prb} in a finite number of IA-ICG outer iterations.
Finally, in view of second remark again, the total number of IA-ICG
outer iterations is as in Theorem~\ref{thm:aicg_compl}(a) but with:
(i) an additional multiplicative factor of $\lceil \log(2m_{2}^{+}/\xi_{0})\rceil $;
and (ii) the constants $m_{1}$ and $M_{2}$ replaced with $(m_{1}+2m_{2}^{+})$
and $(M_{2}+2m_{2}^{+})$, respectively. It is worth mentioning that
a more refined analysis, such as the one in \cite{WJRVarLam2018},
can be applied in order to remove the factor of $\lceil \log(2m_{2}^{+}/\xi_{0})\rceil $
from the previously mentioned complexity.

\subsection{Static and Dynamic DA-ICG Methods}


This subsection presents the static DA-ICG method, but omits the statement of its 
dynamic variant for the sake of brevity. We do argue, however, that
the dynamic variant can be stated in the same way as the dynamic
IA-ICG method of Subsection~\ref{subsec:aicg} but with the call
to the static IA-ICG method replaced with a call to the static DA-ICG
method of this subsection.

We start by stating some additional assumptions. It is assumed that:
\begin{itemize}
\item[(i)] the set $\dom h$ is closed; 
\item[(ii)] there exists a bounded set $\Omega\supseteq\dom h$ for which a projection
oracle exists. 
\end{itemize}
We now state the static DA-ICG method.

\noindent %
\noindent\begin{minipage}[t]{1\columnwidth}%
\noindent \rule[0.5ex]{1\columnwidth}{1pt}

\noindent \textbf{Static DA-ICG Method}

\noindent \rule[0.5ex]{1\columnwidth}{1pt}%
\end{minipage}

\noindent \textbf{Input}: function triple $(f_{1},f_{2},h)$ and scalar
quadruple $(m_{1},M_{1},m_{2},M_{2})\in\r^{4}$ satisfying (A1)--(A4),
tolerance $\hat{\rho}>0$, initial point $\aicgYMin 0\in\dom h$,
and scalar pair $(\lam,\theta)\in\r_{++}\times(0,1)$ satisfying 
\begin{equation}
\lam M_{1}+\theta^{2}\leq\frac{1}{2};\label{eq:d_aicg_lam_restr}
\end{equation}

\noindent \textbf{Output}: a pair $(\hat{\icgY{}},\hat{v})$ satisfying
\eqref{eq:rho_approx_soln} or a \textit{failure} status; 
\begin{itemize}
\item[0.] let $\Delta_{1}(\cdot;\cdot,\cdot)$ be as in \eqref{eq:Delta_def}
with $\mu=1$, and set $A_{0}=0$, $\aicgX 0=\aicgYMin 0$, and $k=1$; 
\item[1.] compute the quantities 
\begin{equation}
\begin{aligned}a_{k-1} & =\frac{1+\sqrt{1+4A_{k-1}}}{2},\quad A_{k}=A_{k-1}+a_{k-1},\\
\aicgXTilde{k-1} & =\frac{A_{k-1}\aicgYMin{k-1}+a_{k-1}\aicgX{k-1}}{A_{k}};
\end{aligned}
\label{eq:accel_d_aicg_def}
\end{equation}
\item[2.] use the R-ACG algorithm to tentatively solve Problem~${\cal B}$
associated with \eqref{eq:icg_subprb}, i.e., with inputs $(\mu,M,\psi_{s},\psi_{n})$
and $(\theta,\acgX 0)$ where the former is as in \eqref{eq:gen_acg_inputs}
and $\acgX 0=\aicgXTilde{k-1}$; if the R-ACG stops with \emph{success},
then let $(\aicgY k,v_{k},\varepsilon_{k})$ denote its output and
go to step~3; otherwise, \textbf{stop} with a \textit{failure} status; 
\item[3.] if the inequality $\Delta_{1}(\icgY{k-1};\aicgY k,v_{k})\leq\varepsilon_{k}$
holds, then go to step~4; otherwise, \textbf{stop} with a \textit{failure}
status; 
\item[4.] set $(\hat{\icgY{}}_{k},\hat{v}_{k})=SRP(\aicgY k,v_{k},\aicgXTilde{k-1})$
where $SRP(\cdot,\cdot,\cdot)$ is described in Subsection~\ref{subsec:prb_of_interest};
if $\|\hat{v}_{k}\|\leq\hat{\rho}$ then \textbf{stop} with a \textit{success}
status and \textbf{output} $(\hat{\icgY{}},\hat{v})=(\hat{\icgY{}}_{k},\hat{v}_{k})$;
otherwise, compute 
\begin{align}
\begin{aligned}\aicgX k & =\argmin_{u\in\Omega}\frac{1}{2}\left\Vert u-[\aicgX{k-1}-a_{k-1}\left(v_{k}+\aicgXTilde{k-1}-\aicgY k\right)]\right\Vert ^{2},\\
\aicgYMin k & =\argmin_{u\in\left\{ \aicgYMin{k-1},\aicgY k\right\} }\left[f_{1}(u)+f_{2}(u)+h(u)\right],
\end{aligned}
\label{eq:z_c_def}
\end{align}
update $k\gets k+1$, and go to step~1. 
\end{itemize}
\noindent \rule[0.5ex]{1\columnwidth}{1pt}

Note that, similar to the static IA-ICG method, the static DA-ICG
method may fail without obtaining a pair satisfying \eqref{eq:rho_approx_soln}.
Proposition~\ref{thm:d_aicg_compl}(c) shows that a sufficient condition
for the method to stop successfully is that $f_{2}$ be convex. Using
arguments similar to the ones employed to derive the dynamic IA-ICG
method, a dynamic version of DA-ICG method can also be developed that
repeatedly invokes the static DA-ICG in place of the static IA-ICG. 

We now make some additional remarks about the above method. First,
it performs two kinds of iterations, namely, ones that are indexed
by $k$ and ones that are performed by the R-ACG algorithm. We refer
to the former kind as outer iterations and the latter kind as inner
iterations. Second, in view of the update for $\aicgYMin k$ in \eqref{eq:z_c_def},
the collection of function values $\{\phi(\aicgYMin i)\}_{i=0}^{k}$
is non-increasing. Third, in view of \eqref{eq:d_aicg_lam_restr},
if $M_{1}>0$ then $0<\lam<(1-2\theta^{2})/(2M_{1})$ whereas if $M_{1}\leq0$
then $0<\lam<\infty$.  Finally, the most expensive part of the method is the R-ACG call in step~2. 
In Section~\ref{sec:spectral_details}, we show that this call
can be replaced with a call to a spectral version of R-ACG that
is dramatically more efficient when the underlying problem 
has the spectral structure as in \eqref{eq:intro_prb}.

It is worth mentioning that the outer iteration scheme of the DA-ICG
method is a monotone and inexact generalization of the A-ECG method
in \cite{nonconv_lan16}. More specifically, this A-ECG method is a
version of the DA-ICG method where: (i) $\theta=0$; (ii) the R-ACG
algorithm in step~2 is replaced by an exact solver of \eqref{eq:icg_subprb};
(iii) the update of $\aicgX k$ in \eqref{eq:z_c_def} is replaced
by an update involving the prox evaluation of the function $a_{k-1}h$;
and (iv) both $f_{1}$ and $f_{2}$ are linearized instead of just
$f_{2}$ in the DA-ICG method. Hence, the DA-ICG method can be significantly
more efficient when its R-ACG call is more efficient than an exact
solver of \eqref{eq:icg_subprb} and/or when the projection onto $\Omega$
is more efficient than evaluating the prox of $a_{k-1}h$.

The next result summarizes some facts about the DA-ICG method. Before
proceeding, we introduce the useful constants 
\begin{gather}
\begin{gathered}D_{h}:=\sup_{u,z\in\dom h}\|u-z\|,\quad D_{\Omega}:=\sup_{u,z\in\Omega}\|u-z\|,\quad\Delta_{\phi}^{0}:=\phi(\aicgYMin 0)-\phi_{*},\\
d_{0}:=\inf_{u^{*}\in{\cal Z}}\{\|\aicgYMin 0-u^{*}\|:\phi(u^{*})=\phi_{*}\},\quad E_{\lam,\theta}:=\sqrt{\lam}L_{1}+\frac{1+\theta\overline{C}_{\lam}}{\sqrt{\lam}}.
\end{gathered}
\label{eq:d_aicg_diam}
\end{gather}

\begin{thm}
\label{thm:d_aicg_compl} The following statements hold about the
static DA-ICG method: 
\begin{itemize}
\item[(a)] it stops in 
\begin{equation}
{\cal O}_{1}\left(\frac{E_{\lam,\theta}^{2}[m_{1}^{+}D_{h}^{2}+\Delta_{\phi}^{0}]}{\hat{\rho}^{2}}+\frac{E_{\lam,\theta}[m_{1}^{+}+1/\lam]^{1/2}D_{\Omega}}{\hat{\rho}}\right)\label{eq:d_aicg_outer_compl}
\end{equation}
outer iterations; 
\item[(b)] if it stops with success, then its output pair $(\hat{\icgY{}},\hat{v})$
is a $\hat{\rho}$--approximate solution of \eqref{eq:main_prb}; 
\item[(c)] if $f_{2}$ is convex, then it always stops with success in 
\begin{equation}
{\cal O}_{1}\left(\frac{E_{\lam,\theta}^{2}m_{1}^{+}D_{h}^{2}}{\hat{\rho}^{2}}+\frac{E_{\lam,\theta}[m_{1}^{+}]^{1/2}D_{\Omega}}{\hat{\rho}}+\frac{E_{\lam,\theta}^{2/3}d_{0}^{2/3}\lam^{-1/3}}{\hat{\rho}^{2/3}}\right)\label{eq:d_aicg_cvx_outer_compl}
\end{equation}
outer iterations. 
\end{itemize}
\end{thm}

We now make three remarks about the above results. First, in the ``best''
scenario of $\max\{m_{1},m_{2}\}\leq0$, i.e., $f_1$ and $f_2$ are convex, we have that \eqref{eq:d_aicg_cvx_outer_compl}
reduces to 
\[
{\cal O}_{1}\left(\left[L_{1}+\frac{1}{\lam}\right]^{2/3}\left[\frac{d_{0}^{2/3}}{\hat{\rho}^{2/3}}\right]\right),
\]
which has a smaller dependence on $\hat{\rho}$ when compared to \eqref{eq:compl_aicg}.
In the ``worst'' scenario of $\min\{m_{1},m_{2}\}>0$, if we take
$\theta={\cal O}(1/\overline{C}_{\lam})$, then \eqref{eq:d_aicg_outer_compl}
reduces to 
\[
{\cal O}_{1}\left(\left[\sqrt{\lam}L_{1}+\frac{1}{\sqrt{\lam}}\right]^{2}\left[\frac{m_{1}^{+}D_{h}^{2}+\phi(\aicgYMin 0)-\phi_{*}}{\hat{\rho}^{2}}\right]\right),
\]
which has the same dependence on $\hat{\rho}$ as in \eqref{eq:compl_aicg}.
Second, part (c) shows that if the method stops with an output pair
$(\hat{\icgY{}},\hat{v})$, regardless of the convexity of $f_{2}$,
then that pair is always an approximate solution of \eqref{eq:main_prb}.
Third, in view of Proposition~\ref{prop:gen_v_hat_rate_d_aicg},
the quantities $L_{1}$ and $\overline{C}_{\lam}$ in all of the previous
complexity results can be replaced by their averaged counterparts
in \eqref{eq:d_avg_def}. As these averaged quantities only depend
on $\{(\aicgY i,\hat{\icgY{}}_{i},\aicgXTilde{i-1})\}_{i=1}^{k}$,
we can infer that the static DA-ICG method, like the static IA-ICG
method of the previous subsection, also adapts to the local geometry
of its input functions.

\section{Exploiting the Spectral Decomposition}

\label{sec:spectral_details} Recall that at every outer iteration
of the ICG methods in Section~\ref{sec:accelerated_icg}, a call
to the R-ACG algorithm is made to tentatively solve Problem~${\cal B}$
(see Subsection~\ref{subsec:prb_of_interest}) associated with \eqref{eq:icg_subprb}.
Our goal in this section is to present a more efficient version of R-ACG (based on the idea outlined in Section~\ref{sec:intro})
 when the underlying problem has the spectral structure as in \eqref{eq:intro_prb}.

The content of this section is divided into two subsections. The first
one presents the aforementioned algorithm, whereas the second one
proves its key properties.

\subsection{\label{subsec:spectral_exploit}Spectral R-ACG Algorithm}

This subsection presents the R-ACG algorithm mentioned above. Throughout our presentation,
we let $\acgMatX 0$ represent the starting point given to the R-ACG
algorithm by the two ICG methods.

We first state the aforementioned efficient algorithm.

\noindent %
\noindent\begin{minipage}[t]{1\columnwidth}%
\noindent \rule[0.5ex]{1\columnwidth}{1pt}

\noindent \textbf{Spectral R-ACG Algorithm}

\noindent \rule[0.5ex]{1\columnwidth}{1pt}%
\end{minipage}

\noindent \textbf{Input}: a quadruple $(M_{2},f_{1},f_{2}^{{\cal V}},h^{{\cal V}})$
satisfying (A1)--(A3) with $(f_{2},h)=(f_{2}^{{\cal V}},h^{{\cal V}})$
and a triple $(\lam,\theta,\acgMatX 0)$;

\noindent \textbf{Output}: a triple $(\icgMatY{},V,\varepsilon)$
that solves Problem~${\cal B}$ associated with \eqref{eq:icg_subprb}
or a \textit{failure} status; 
\begin{itemize}
\item[1.] compute 
\begin{equation}
\begin{gathered}\acgMatX 0^{\lam}:=\acgMatX 0-\lam\nabla f_{1}(\acgMatX 0),\end{gathered}
\label{eq:SVD_quants}
\end{equation}
and a pair $(P,Q)\in{\cal U}^{m}\times{\cal U}^{n}$ satisfying $\acgMatX 0^{\lam}=P[\dg\sigma(\acgMatX 0^{\lam})]Q^{*}$; 
\item[2.] use the R-ACG algorithm to tentatively solve Problem~${\cal B}$
associated with \eqref{eq:icg_vec_prox}, i.e., with inputs $(\mu,M,\psi_{s}^{{\cal V}},\psi_{n}^{{\cal V}})$
and $(\theta,\acgX 0)$ where the former is given by 
\begin{equation}
\begin{gathered}\mu=1,\quad M:=\lam M_{2}^{+}+1,\\
\psi_{s}^{{\cal V}}:=\lam f_{2}^{{\cal V}}-\inner{\sigma(\acgMatX 0^{\lam})}{\cdot}+\frac{1}{2}\|\cdot\|^{2},\quad\psi_{n}^{{\cal V}}:=\lam h^{{\cal V}},
\end{gathered}
\label{eq:vec_acg_inputs}
\end{equation}
and $\acgX 0=\Dg(P^{*}\acgMatX 0Q)$; if the R-ACG stops with \emph{success},
then let $(\icgY{},v,\varepsilon)$ denote its output and go to step~3;
otherwise, \textbf{stop} with a \textit{failure} status; 
\item[3.] set $\icgMatY{}=P(\dg\icgY{})Q^{*}$ and $V=P(\dg v)Q^{*}$, and
output the triple $(\icgMatY{},V,\varepsilon)$. 
\end{itemize}
\noindent \rule[0.5ex]{1\columnwidth}{1pt}

We now make three remarks about the above algorithm. First, the matrices
$P$ and $Q$ in step~1 can be obtained by computing an SVD of $\acgMatX 0^{\lam}$.
Second, in view of Proposition~\ref{prop:acg_key_props}(a) and the
fact that $(\mu,M)$ in \eqref{eq:vec_acg_inputs} and \eqref{eq:gen_acg_inputs}
are the same, the iteration complexity is the same as the vanilla
R-ACG algorithm. Finally, because the functions $\psi_{s}^{{\cal V}}$
and $\psi_{n}^{{\cal V}}$ in \eqref{eq:vec_acg_inputs} have vector
inputs over $\r^{r}$, the steps in the spectral R-ACG algorithm are
significantly less costly than the ones in the R-ACG algorithm, which
use functions with matrix inputs over $\r^{m\times n}$.

The following result, whose proof is in the next subsection, presents
the key properties of this algorithm.
\begin{prop}
\label{prop:acg_implementation} The spectral R-ACG algorithm has
the following properties: 
\begin{itemize}
\item[(a)] if it stops with success, then its output triple $(\icgMatY{},V,\varepsilon)$
solves Problem~${\cal B}$ associated with \eqref{eq:icg_subprb}; 
\item[(b)] if $f_{2}$ is convex, then it always stops with success and its
output $(\icgMatY{},V,\varepsilon)$ solves Problem~${\cal A}$ associated
with \eqref{eq:icg_subprb}. 
\end{itemize}
\end{prop}

\subsection{Proof of Proposition~\ref{prop:acg_implementation}}

For the sake of brevity, let $(\psi_{s},\psi_{n})$ be as in \eqref{eq:gen_acg_inputs}
and, using $P$ and $Q$ from the spectral R-ACG algorithm, define
for every $(u,U)\in\r^{r}\times\r^{m\times n}$, the functions 
\begin{gather*}
{\cal M}(u):=P(\dg u)Q^{*},\quad{\cal V}(U):=\Dg(P^{*}UQ),\\
\psi(U):=\psi_{s}(U)+\psi_{n}(U),\quad\psi^{{\cal V}}(u):=\psi_{s}^{{\cal V}}(u)+\psi_{n}^{{\cal V}}(u).
\end{gather*}
The first result relates $(\psi_{s},\psi_{n})$ to $(\psi_{s}^{{\cal V}},\psi_{n}^{{\cal V}})$.
\begin{lem}
\label{lem:psi_spec_props} Let $(\icgY{},v,\varepsilon)$ and $(\icgMatY{},V)$
be as in the spectral R-ACG algorithm. Then, the following properties
hold: 
\begin{itemize}
\item[(a)] we have 
\[
\psi_{n}^{{\cal V}}(\icgY{})=\psi_{n}(\icgMatY{}),\quad\psi_{s}^{{\cal V}}(\icgY{})+B_{0}^{\lam}=\psi_{s}(\icgMatY{}),
\]
where $B_{0}^{\lam}:=\lam f_{1}(\acgMatX 0)-\lam\inner{\nabla f_{1}(\acgMatX 0)}{\acgMatX 0}+\|{\acgMatX 0}\|_{F}^{2}/2$; 
\item[(b)] we have 
\begin{equation}
V\in\pt_{\varepsilon}\left(\psi-\frac{1}{2}\|\cdot-\icgMatY{}\|_{F}^{2}\right)(\icgMatY{})\iff v\in\pt_{\varepsilon}\left(\psi^{{\cal V}}-\frac{1}{2}\|\cdot-\icgY{}\|^{2}\right)(\icgY{}).\label{eq:mat_vec_spec_incl}
\end{equation}
\end{itemize}
\end{lem}

\begin{proof}
(a) The relationship between $\psi_{n}^{{\cal V}}$ and $\psi_{n}$
is immediate. On the other hand, using the definitions of $\icgMatY{},f_{2}$,
and $B_{0}^{\lam}$, we have 
\begin{align*}
 & \psi_{s}^{{\cal V}}(\icgY{})+B_{0}^{\lam}=\lam f_{2}(\icgMatY{})-\inner{\acgMatX 0^{\lam}}{\icgMatY{}}+\frac{1}{2}\|\icgMatY{}\|_{F}^{2}+B_{0}^{\lam}\\
 & =\lam\left[f_{2}(\icgMatY{})+f_{1}(\acgMatX 0)+\inner{\nabla f_{1}(\acgMatX 0)}{\icgMatY{}-\acgMatX 0}\right]+\frac{1}{2}\|\icgMatY{}-\acgMatX 0\|_{F}^{2}=\psi_{s}(\icgMatY{}).
\end{align*}

(b) Let $S_{0}=V+\acgMatX 0^{\lam}-\icgMatY{}$ and $s_{0}=v+\sigma(\acgMatX 0^{\lam})-\icgY{}$,
and note that $S_{0}={\cal M}(s_{0})$. Moreover, in view of part
(a) and the definition of $\psi$, observe that the left inclusion
in \eqref{eq:mat_vec_spec_incl} is equivalent to $S_{0}\in\pt_{\varepsilon}(\lam[f_{2}+h])(\icgMatY{})$.
Using this observation, the fact that $S_{0}$ and $\icgMatY{}$ have
a simultaneous SVD, and Theorem~\ref{thm:spectral_approx_subdiff}
with $(S,s)=(S_{0},s_{0})$, $\Psi=\lam[f_{2}+h]$, and $\Psi^{{\cal V}}=\lam[f_{2}^{{\cal V}}+h^{{\cal V}}]$,
we have that the left inclusion in \eqref{eq:mat_vec_spec_incl} is
also equivalent to $s_{0}\in\pt_{\varepsilon}(\lam[f_{2}^{{\cal V}}+h^{{\cal V}}])(\icgY{})$.
The conclusion now follows from the observation that the latter inclusion
is equivalent to the the right inclusion in \eqref{eq:mat_vec_spec_incl}. 
\end{proof}
We are now ready to give the proof of Proposition~\ref{prop:acg_implementation}.
\begin{proof}[Proof of Proposition~\ref{prop:acg_implementation}]
(a) Since $(\icgY{},v)=({\cal V}(\icgMatY{}),{\cal V}(V))$, notice
that the successful termination of the algorithm implies that the
inequality in \eqref{eq:cvx_inexact} and \eqref{eq:prb_B_Delta_ineq}
hold. Using this remark, the fact that $\|V\|_{F}^{2}=\|v\|^{2}$,
and the bound 
\begin{align}
 & \theta^{2}\|\acgX j-\acgX 0\|^{2}=\theta^{2}\left(\|\acgX j\|^{2}-2\inner{\acgX j}{{\cal V}(\acgX 0)}+\|\acgMatX 0\|_{F}^{2}\right)+\theta^{2}(\|{\cal V}(\acgX 0)\|^{2}-\|\acgMatX 0\|_{F}^{2})\nonumber \\
 & \leq\theta^{2}\left(\|\acgMatX j\|^{2}_F-2\inner{\acgMatX j}{\acgMatX 0}+\|\acgMatX 0\|_{F}^{2}\right)=\theta^{2}\|\acgMatX j-\acgMatX 0\|_{F}^{2},\label{eq:hpe_equiv}
\end{align}
we then have that the inequality in \eqref{eq:cvx_inexact} also holds
with $(\icgY{},v)=(\icgMatY{},V)$.

To show the corresponding inequality for \eqref{eq:prb_B_Delta_ineq},
let $(\icgMatY r,V_{r})=RP(\icgMatY{},V)$ using the refinement procedure
in Section~\ref{sec:background}. Moreover, let $(\icgY r,v_{r})=RP(\icgY{},v)$
and $\Delta_{1}^{{\cal V}}(\cdot;\cdot,\cdot)$ be as in \eqref{eq:Delta_def},
where $(\psi_{s},\psi_{n})=(\psi_{s}^{{\cal V}},\psi_{n}^{{\cal V}})$.
It now follows from \eqref{eq:z_hat_def}, \eqref{eq:v_hat_def},
Lemma~\ref{lem:spec_prox} with $\Psi=\psi_{n}$ and $S=V+M\icgMatY{}-\nabla\psi_{s}(\icgMatY{})$,
and Lemma~\ref{lem:spec_prop}(b) that $\icgMatY r,\icgMatY{},V$,
and $V_{r}$ have a simultaneous SVD. As a consequence of this, the
first remark, and Lemma~\ref{lem:psi_spec_props}(a), we have that
\begin{align*}
 & \varepsilon\geq\Delta_{1}^{{\cal V}}(\icgY r;\icgY{},v)=\psi^{{\cal V}}(\icgY{})-\psi^{{\cal V}}(\icgY r)-\inner v{\icgY{}-\icgY r}+\frac{1}{2}\|\icgY r-\icgY{}\|^{2}\\
 & =\psi(\icgMatY{})-\psi(\icgMatY r)-\inner V{\icgMatY{}-\icgMatY r}+\frac{1}{2}\|\icgMatY r-\icgMatY{}\|^{2}=\Delta_{1}(\icgMatY r;\icgMatY{},V), 
\end{align*}
and hence that \eqref{eq:prb_B_Delta_ineq} holds with $(\icgY{},v)=(\icgMatY{},V)$.

(b) This follows from part (a), Proposition~\ref{prop:acg_properties}(c),
and Lemma~\ref{lem:psi_spec_props}(b). 
\end{proof}

\section{\label{sec:computational}Computational Results}

This section presents computational results that highlight the performance
of the dynamic IA-ICG and dynamic DA-ICG methods, and it contains three subsections.
The first one describes the implementation details, the second presents
computational results related to a set of spectral composite problem,
while the third gives some general comments about the computational
results.

\subsection{Implementation Details}

This subsection precisely describes the implementation of the methods
and experiments of this section. Moreover, all of the code needed
to replicate these experiments is readily available online\footnote{See \href{https://github.com/wwkong/nc_opt/tree/master/tests/papers/icg}{https://github.com/wwkong/nc\_opt/tree/master/tests/papers/icg}.}.

We first describe some practical modifications to the dynamic IA-ICG method.
Given $\lam>0$ and $(\acgX j,\acgX 0)\in{\cal Z}^{2}$, denote 
\begin{align*}
\Delta_{\phi}^{\lam} & =4\lam\left[\phi(\acgX 0)-\widetilde{\ell}_{\phi}(\acgX j;\acgX 0)-\frac{M_{1}}{2}\|\acgX j-\acgX 0\|^{2}\right]
\end{align*}
where $\widetilde{\ell}_{\phi}$ is as in \eqref{eq:ell_phi_C_bar_lam_def}.
Motivated by the first inequality in the descent condition \eqref{eq:aicg_descent},
we relax \eqref{eq:acg_ineq} in the R-ACG call to the three separate
conditions: $\|\acgX j-\acgX 0\|^{2}\leq\Delta_{\phi}^{\lam},\|\acgU j\|^{2}\leq\Delta_{\phi}^{\lam}$,
and $2\eta_{j}\leq\Delta_{\phi}^{\lam}$.

We now describe some modifications and parameter choices that are
common to both methods. First, both ICG methods use the spectral R-ACG
algorithm of Subsection~\ref{subsec:spectral_exploit} in place of
the R-ACG algorithm of Section~\ref{sec:background}. Moreover, this
R-ACG variant uses a line search subroutine for estimating the upper
curvature $M$ that is used during its execution. Second, when each
of the dynamic ICG methods invokes their static counterparts, the
parameters $A_{0}$ and $\icgY 0$ are set to be the last obtained
parameters of the previous invocation or the original parameters if
it is the first invocation, i.e., we implement a warm--start strategy.
Third, we adaptively update $\lam$ at each outer iteration as follows:
given the old value of $\lam=\lam_{{\rm old}}$ at the $k^{{\rm th}}$
outer iteration, the new value of $\lam=\lam_{{\rm new}}$ at the
$(k+1)^{{\rm th}}$ iteration is given by 
\[
\lam_{{\rm new}}=\begin{cases}
\lam_{{\rm old}}, & r_{k}\in\left[0.5,2.0\right],\\
\lam_{{\rm old}}\cdot\sqrt{0.5}, & r_{k}<0.5,\\
\lam_{{\rm old}}\cdot\sqrt{2}, & r_{k}>2.0,
\end{cases}\quad r_{k}=\frac{\left[\lam(M_{2}^{+}+2m_{2}^{+})+1\right]\|\icgY k-\hat{\icgY{}}_{k}\|}{\|\hat{v}_{k}-\left[\lam(M_{2}^{+}+2m_{2}^{+})+1\right](\icgY k-\hat{\icgY{}}_{k})\|}.
\]
Fourth, we take $\mu=1/2$ rather than $\mu=1$ for each of R-ACG
calls in order to reduce the possibility of a failure from the R-ACG
algorithm. Fifth, in view of \eqref{eq:hpe_equiv}, we relax condition
\eqref{eq:acg_ineq} in the vector-based R-ACG call of Subsection~\ref{subsec:spectral_exploit}
to 
\[
\|\acgU j\|^{2}+2\eta_{j}\leq\theta^{2}\|\acgX j-\acgX 0\|^{2}+\tau,
\]
where $\tau:=\theta^{2}(\|\acgMatX 0\|_{F}^{2}-\|\acgX 0\|^{2})\geq0$.
Finally, both ICG methods choose the common hyperparameters $(\xi_{0},\lam,\theta)=(M_{1},5/M_{1},1/2)$
at initialization.

We now describe the five other benchmark methods considered. Throughout
their descriptions, we let $m=m_{1}+m_{2}$, $M=M_{1}+M_{2}$, and
$L=\max\{m,M\}$. The first method is Nesterov's efficient ECG method
of \cite{nesterov2007gradient} with $(\lam,\gamma_{u},\gamma_{d})=(100/L,2,2)$.
The second method is the accelerated inexact proximal point (AIPP)
method of \cite{WJRVarLam2018} with $(\lam,\theta,\tau)=(1/m,4,10[\lam M+1])$
and the R-AIPPv2 stepsize scheme. The third method is a variant of
the A-ECG method of \cite[Algorithm 2]{nonconv_lan16}, which we abbreviate
as AG. In particular, this variant chooses its parameters as in \cite[Corollary 2]{nonconv_lan16}
with $L_{\Psi}$ replaced by $M$, i.e., $\beta_{k}=1/(2M)$ for every
$k$ (implying a more aggressive stepsize policy). It is worth mentioning
that we tested the more conservative AG variant with $\beta_{k}=1/(2L_{\Psi})$
and observed that it performed substantially less efficient than the
above aggressive variant. The fourth method is a special implementation
of the adaptive A-ECG method in \cite{LanUniformly} with $(\gamma_{1},\gamma_{2},\gamma_{3})=(0.4,0.4,1.0)$
and $(\delta,\sigma)=(10^{-2},10^{-10})$, which we abbreviate as
UP. More specifically, we consider the UPFAG-fullBB method described
in \cite[Section 4]{LanUniformly}, which uses a Barzilai-Borwein
type stepsize selection strategy. The last is the A-ECG method of \cite{Liang2019},
named NC-FISTA, with $(\xi,\lam)=(1.05m,0.99/M)$, which we abbreviate
as NCF.

Finally, we state some additional details about the numerical experiments.
First, the problems considered are of the form in \eqref{eq:intro_prb}
and satisfy assumptions (A1)--(A4) with $f_{2}=f_{2}^{{\cal V}}\circ\sigma$
and $h=h^{{\cal V}}\circ\sigma$. Second, given a tolerance $\hat{\rho}>0$
and an initial point $\icgMatY 0\in\dom h$, every method in this
section seeks a pair $(\hat{\icgMatY{}},\hat{V})\in\dom h\times\r^{m\times n}$
satisfying 
\[
\hat{V}\in\nabla f_{1}(\hat{\icgMatY{}})+\nabla(f_{2}^{{\cal V}}\circ\sigma)(\hat{\icgMatY{}})+\pt(h^{{\cal V}}\circ\sigma)(\hat{\icgMatY{}}),\quad\frac{\|\hat{V}\|}{\|\nabla f_{1}(\icgMatY 0)+(f_{2}^{{\cal V}}\circ\sigma)(\icgMatY 0)\|+1}\leq\hat{\rho},
\]
and stops after 1000 seconds if such a point cannot be found. Third,
to be concise, we abbreviate the IA-ICG and DA-ICG methods as IA and
DA, respectively. Finally, all described algorithms are implemented
in MATLAB 2020a and are run on Linux 64-bit machines that contain
at least 8 GB of memory.

\subsection{Spectral Composite Problems}

This subsection presents computational results of a set of spectral
composite optimization problems and contains two sub-subsections.
The first one examines a class of nonconvex matrix completion problems,
while the second one examines a class of blockwise matrix completion
problems.

\subsubsection{Matrix completion}

\label{subsec:mc}

Given a quadruple $(\alpha,\beta,\mu,\theta)\in\r_{++}^{4}$, a data
matrix $A\in\r^{\ell\times n}$, and indices $\Omega$, this subsection
considers the following constrained matrix completion (MC) problem:
\begin{align*}
\begin{aligned}\min_{U\in\r^{m\times n}}\  & \frac{1}{2}\|P_{\Omega}(U-A)\|_{F}^{2}+\kappa_{\mu}\circ\sigma(U)+\tau_{\alpha}\circ\sigma(U)\\
\text{s.t.}\  & \|U\|_{F}^{2}\leq\sqrt{\ell n}\cdot\max_{i,j}|A_{ij}|,
\end{aligned}
\end{align*}
where $P_{\Omega}$ is the linear operator that zeros out any entry
that is not in $\Omega$ and 
\begin{align*}
\kappa_{\mu}(z)=\frac{\mu\beta}{\theta}\sum_{i=1}^{n}\log\left(1+\frac{|z_{i}|}{\theta}\right),\quad\tau_{\alpha}(z)=\alpha\beta\left[1-\exp\left(-\frac{\|z\|_{2}^{2}}{2\theta}\right)\right]
\end{align*}
for every $z\in\r^{n}$. Here, the function $\kappa_{\mu}+\tau_{\alpha}$
is a nonconvex generalization of the convex elastic net regularizer
(see, for example, \cite{Sun2012}), and it is well-known (see, for
example, \cite{Yao2017}) that the function $\kappa_{\mu}-\mu\|\cdot\|_{*}$
is concave, differentiable, and has a $(2\beta\mu/\theta^{2})$-Lipschitz
continuous gradient.

We now describe the different data matrices that are considered. Each
matrix $A\in\r^{\ell\times n}$ is obtained from a different collaborative
filtering system where each row represents a unique user, each column
represents a unique item, and each entry represents a particular rating.
Table~\ref{tabl:mc_data_mat} lists the names of each data set, where
the data originates from (in the footnotes), and some basic statistics
about the matrices.

\begin{table}
\centering{}%
\begin{tabular}{|>{\centering}m{2.7cm}|>{\centering}m{1cm}>{\centering}m{1cm}|>{\centering}m{1.7cm}|>{\centering}m{1.7cm}>{\centering}m{1.7cm}|}
\hline 
{\small{}Name} & {\small{}$\ell$ } & {\small{}$n$ } & {\small{}\% nonzero } & {\small{}$\min_{i,j}A_{ij}$ } & {\small{}$\max_{i,j}A_{ij}$}\tabularnewline
\hline 
{\footnotesize{}Anime}\tablefootnote{See the subset of the ratings from \url{https://www.kaggle.com/CooperUnion/anime-recommendations-database}
where each user has rated at least 720 items.}  & {\footnotesize{}506}  & {\footnotesize{}9437}  & {\footnotesize{}10.50\%}  & {\footnotesize{}1}  & {\footnotesize{}10}\tabularnewline
{\footnotesize{}FilmTrust}\tablefootnote{See the ratings in the file ``ratings.txt'' under the FilmTrust
section in \url{https://www.librec.net/datasets.html}.}  & {\footnotesize{}1508}  & {\footnotesize{}2071}  & {\footnotesize{}1.14\%}  & {\footnotesize{}0.5}  & {\footnotesize{}8}\tabularnewline
\hline 
\end{tabular}\caption{Description of the MC data matrices $A\in\protect\r^{m\times n}$.}
\label{tabl:mc_data_mat} 
\end{table}

We now describe the experiment parameters considered. First the starting
point $Z_{0}$ is randomly generated from a shifted binomial distribution
that closely follows the data matrix $A$. More specifically, the
entries of $Z_{0}$ are distributed according to a $\textsc{Binomial}(a,\mu/a)-\underline{A}$
distribution, where $\mu$ is the sample average of the nonzero entries
in $A$, the integer $a$ is the ceiling of the range of ratings in
$A$, and $\underline{A}$ is the minimum rating in $A$. Second,
the decomposition of the objective function is as follows 
\begin{gather}
f_{1}=\frac{1}{2}\|P_{\Omega}(\cdot-A)\|_{F}^{2},\quad f_{2}^{{\cal V}}=\mu\left[\kappa_{\mu}(\cdot)-\frac{\beta}{\theta}\|\cdot\|_{1}\right]+\tau_{\alpha}(\cdot),\quad h^{{\cal V}}=\frac{\mu\beta}{\theta}\|\cdot\|_{1}+\delta_{{\cal F}}(\cdot),\label{eq:mc_fn_decomp}
\end{gather}
where ${\cal F}=\{U\in\r^{m\times n}:\|U\|_{F}\leq\sqrt{\ell n}\cdot\max_{i,j}|A_{ij}|\}$
is the set of feasible solutions. Third, in view of the previous decomposition,
the curvature parameters are set to be 
\begin{equation}
m_{1}=0,\quad M_{1}=1,\quad m_{2}=\frac{2\beta\mu}{\theta^{2}}+\frac{2\alpha\beta}{\theta}\exp\left(\frac{-3\theta}{2}\right),\quad M_{2}=\frac{\alpha\beta}{\theta},\label{eq:mc_curv_param}
\end{equation}
where it can be shown that the smallest and largest eigenvalues of
$\nabla^{2}\tau_{\alpha}(z)$ are bounded below and above by $-2\alpha\beta\exp(-3\theta/2)/\theta$
and $\alpha\beta/\theta$, respectively, for every $z\in\r^{n}$.
Finally, each problem instance uses a specific data matrix $A$ from
Table~\ref{tabl:mc_data_mat}, the hyperparameters $(\alpha,\beta,\mu)=(10,20,2)$
and $\hat{\rho}=10^{-6}$, different values of the parameter $\theta$,
and $\Omega$ to be the index set of nonzero entries in the chosen
matrix $A$. 

We now present the results. Figure~\ref{fig:mc_graphs} presents
two subplots for the results of the Anime dataset under a value of
$\theta=10^{-1}$. The first subplot contains the log objective value
against runtime, while the second one contains the log of the minimal
subgradients, i.e. $\min_{i\leq k}\|\hat{V}_{i}\|$, against runtime.
Tables~\ref{tabl:mc_anime} to \ref{tabl:mc_filmtrust} present the
minimal subgradient size obtained within the time limit of 1000. Moreover,
each row of these tables corresponds to a different choice of $\theta$
and the bolded numbers highlight which algorithm performed the best
in terms of the size obtained in a run.

\begin{figure}[tbhp]
\begin{centering}
\includesvg[width=1\linewidth]{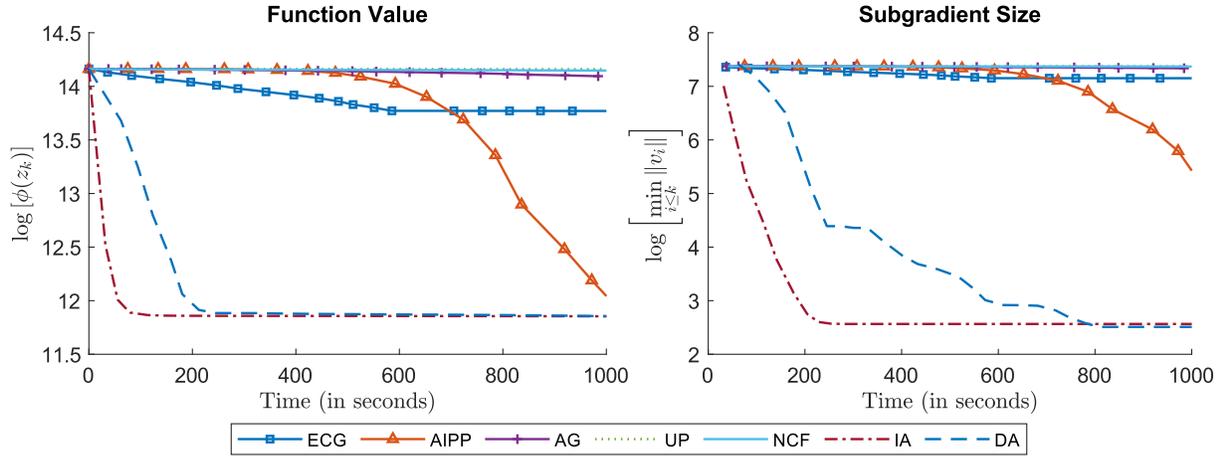}
\par\end{centering}
\caption{Function values and minimum subgradients for the Anime dataset with
$\theta=10^{-1}$.}
\label{fig:mc_graphs} 
\end{figure}

\begin{table}[tbh]
\begin{centering}
\begin{tabular}{|>{\centering}m{1.8cm}|>{\centering}m{0.8cm}|>{\centering}p{1cm}>{\centering}p{1cm}>{\centering}p{1cm}>{\centering}p{1cm}>{\centering}p{1cm}>{\centering}p{1cm}>{\centering}p{1cm}|}
\hline 
{\footnotesize{}Parameters} & {\footnotesize{}Time} & \multicolumn{7}{c|}{{\footnotesize{}Minimum Subgradient Size ($\min_{i\leq k}\|\hat{V}_{i}\|$)}}\tabularnewline
\cline{3-9} \cline{4-9} \cline{5-9} \cline{6-9} \cline{7-9} \cline{8-9} \cline{9-9} 
{\footnotesize{}$(\theta,m,M)^{T}$} & {\footnotesize{}$t$} & {\footnotesize{}ECG} & {\footnotesize{}AIPP} & {\footnotesize{}AG} & {\footnotesize{}UP} & {\footnotesize{}NCF} & {\footnotesize{}IA} & {\footnotesize{}DA}\tabularnewline
\hline 
\multirow{4}{1.8cm}{{\scriptsize{}$\left[\begin{array}{c}
1\\
169\\
201
\end{array}\right]$}} & {\scriptsize{}100} & {\scriptsize{}1088.6} & {\scriptsize{}1421.0} & {\scriptsize{}1568.9} & {\scriptsize{}1599.4} & {\scriptsize{}1488.8} & \textbf{\scriptsize{}13.0} & {\scriptsize{}78.5}\tabularnewline
 & {\scriptsize{}200} & {\scriptsize{}1088.6} & {\scriptsize{}221.9} & {\scriptsize{}1510.2} & {\scriptsize{}132.6} & {\scriptsize{}1362.4} & \textbf{\scriptsize{}11.6} & {\scriptsize{}39.2}\tabularnewline
 & {\scriptsize{}400} & {\scriptsize{}1088.6} & {\scriptsize{}55.6} & {\scriptsize{}1284.6} & \textbf{\scriptsize{}7.5} & {\scriptsize{}1147.9} & {\scriptsize{}11.6} & {\scriptsize{}11.1}\tabularnewline
 & {\scriptsize{}800} & {\scriptsize{}1088.6} & {\scriptsize{}7.7} & {\scriptsize{}716.7} & \textbf{\scriptsize{}7.5} & {\scriptsize{}862.7} & {\scriptsize{}11.6} & {\scriptsize{}11.1}\tabularnewline
\hline 
\multirow{4}{1.8cm}{{\scriptsize{}$\left[\begin{array}{c}
0.1\\
11443\\
2001
\end{array}\right]$}} & {\scriptsize{}100} & {\scriptsize{}1542.0} & {\scriptsize{}1595.8} & {\scriptsize{}1593.8} & {\scriptsize{}-} & {\scriptsize{}1595.0} & \textbf{\scriptsize{}189.7} & {\scriptsize{}1345.1}\tabularnewline
 & {\scriptsize{}200} & {\scriptsize{}1489.9} & {\scriptsize{}1595.0} & {\scriptsize{}1591.5} & {\scriptsize{}1595.2} & {\scriptsize{}1594.2} & \textbf{\scriptsize{}23.1} & {\scriptsize{}378.1}\tabularnewline
 & {\scriptsize{}400} & {\scriptsize{}1391.0} & {\scriptsize{}1587.8} & {\scriptsize{}1584.3} & {\scriptsize{}1595.1} & {\scriptsize{}1592.3} & \textbf{\scriptsize{}13.0} & {\scriptsize{}60.6}\tabularnewline
 & {\scriptsize{}800} & {\scriptsize{}1276.3} & {\scriptsize{}990.5} & {\scriptsize{}1557.2} & {\scriptsize{}1594.3} & {\scriptsize{}1589.2} & \textbf{\scriptsize{}13.0} & {\scriptsize{}13.0}\tabularnewline
\hline 
\multirow{4}{1.8cm}{{\scriptsize{}$\left[\begin{array}{c}
0.01\\
839400\\
20001
\end{array}\right]$}} & {\scriptsize{}100} & {\scriptsize{}1594.6} & {\scriptsize{}1595.9} & {\scriptsize{}1595.6} & {\scriptsize{}1595.8} & {\scriptsize{}1595.9} & \textbf{\scriptsize{}162.9} & {\scriptsize{}452.0}\tabularnewline
 & {\scriptsize{}200} & {\scriptsize{}1592.8} & {\scriptsize{}1595.6} & {\scriptsize{}1595.0} & {\scriptsize{}1595.8} & {\scriptsize{}1595.8} & \textbf{\scriptsize{}33.5} & {\scriptsize{}68.0}\tabularnewline
 & {\scriptsize{}400} & {\scriptsize{}1589.8} & {\scriptsize{}1569.5} & {\scriptsize{}1592.2} & {\scriptsize{}1595.8} & {\scriptsize{}1595.8} & {\scriptsize{}15.3} & \textbf{\scriptsize{}14.7}\tabularnewline
 & {\scriptsize{}800} & {\scriptsize{}1583.8} & {\scriptsize{}861.8} & {\scriptsize{}1582.3} & {\scriptsize{}1595.7} & {\scriptsize{}1595.7} & {\scriptsize{}15.3} & \textbf{\scriptsize{}14.1}\tabularnewline
\hline 
\end{tabular}
\par\end{centering}
\caption{Minimum subgradient sizes for the Anime dataset. Times are in seconds
and \textquotedblleft -\textquotedblright{} indicates a run that did
not generate a subgradient within the given time limit.}
\label{tabl:mc_anime}
\end{table}

\begin{table}[tbh]
\begin{centering}
\begin{tabular}{|>{\centering}m{1.8cm}|>{\centering}m{0.8cm}|>{\centering}p{1cm}>{\centering}p{1cm}>{\centering}p{1cm}>{\centering}p{1cm}>{\centering}p{1cm}>{\centering}p{1cm}>{\centering}p{1cm}|}
\hline 
{\footnotesize{}Parameters} & {\footnotesize{}Time} & \multicolumn{7}{c|}{{\footnotesize{}Minimum Subgradient Size ($\min_{i\leq k}\|\hat{V}_{i}\|$)}}\tabularnewline
\cline{3-9} \cline{4-9} \cline{5-9} \cline{6-9} \cline{7-9} \cline{8-9} \cline{9-9} 
{\footnotesize{}$(\theta,m,M)^{T}$} & {\footnotesize{}$t$} & {\footnotesize{}ECG} & {\footnotesize{}AIPP} & {\footnotesize{}AG} & {\footnotesize{}UP} & {\footnotesize{}NCF} & {\footnotesize{}IA} & {\footnotesize{}DA}\tabularnewline
\hline 
\multirow{4}{1.8cm}{{\scriptsize{}$\left[\begin{array}{c}
1\\
169\\
201
\end{array}\right]$}} & {\scriptsize{}100} & {\scriptsize{}127.1} & {\scriptsize{}328.6} & {\scriptsize{}328.5} & {\scriptsize{}-} & {\scriptsize{}326.2} & \textbf{\scriptsize{}77.7} & {\scriptsize{}342.4}\tabularnewline
 & {\scriptsize{}200} & {\scriptsize{}106.7} & {\scriptsize{}326.2} & {\scriptsize{}326.8} & {\scriptsize{}330.0} & {\scriptsize{}319.4} & \textbf{\scriptsize{}60.7} & {\scriptsize{}203.1}\tabularnewline
 & {\scriptsize{}400} & {\scriptsize{}106.7} & {\scriptsize{}294.6} & {\scriptsize{}319.2} & {\scriptsize{}330.0} & {\scriptsize{}305.9} & \textbf{\scriptsize{}60.7} & {\scriptsize{}186.4}\tabularnewline
 & {\scriptsize{}800} & {\scriptsize{}106.7} & {\scriptsize{}107.4} & {\scriptsize{}291.0} & {\scriptsize{}251.9} & {\scriptsize{}280.5} & \textbf{\scriptsize{}60.7} & {\scriptsize{}186.4}\tabularnewline
\hline 
\multirow{4}{1.8cm}{{\scriptsize{}$\left[\begin{array}{c}
0.1\\
11443\\
2001
\end{array}\right]$}} & {\scriptsize{}100} & {\scriptsize{}309.0} & {\scriptsize{}330.0} & {\scriptsize{}329.6} & {\scriptsize{}329.9} & {\scriptsize{}329.9} & \textbf{\scriptsize{}71.0} & {\scriptsize{}242.3}\tabularnewline
 & {\scriptsize{}200} & {\scriptsize{}287.0} & {\scriptsize{}326.9} & {\scriptsize{}327.8} & {\scriptsize{}329.9} & {\scriptsize{}329.5} & \textbf{\scriptsize{}71.0} & {\scriptsize{}235.4}\tabularnewline
 & {\scriptsize{}400} & {\scriptsize{}248.0} & {\scriptsize{}188.7} & {\scriptsize{}321.9} & {\scriptsize{}329.8} & {\scriptsize{}328.8} & \textbf{\scriptsize{}71.0} & {\scriptsize{}202.7}\tabularnewline
 & {\scriptsize{}800} & {\scriptsize{}186.9} & {\scriptsize{}188.7} & {\scriptsize{}301.8} & {\scriptsize{}329.4} & {\scriptsize{}327.4} & \textbf{\scriptsize{}71.0} & {\scriptsize{}202.7}\tabularnewline
\hline 
\multirow{4}{1.8cm}{{\scriptsize{}$\left[\begin{array}{c}
0.01\\
839400\\
20001
\end{array}\right]$}} & {\scriptsize{}100} & {\scriptsize{}330.1} & {\scriptsize{}330.2} & {\scriptsize{}330.2} & {\scriptsize{}-} & {\scriptsize{}330.2} & \textbf{\scriptsize{}91.8} & {\scriptsize{}263.9}\tabularnewline
 & {\scriptsize{}200} & {\scriptsize{}330.0} & {\scriptsize{}330.2} & {\scriptsize{}330.2} & {\scriptsize{}330.2} & {\scriptsize{}330.2} & \textbf{\scriptsize{}91.8} & {\scriptsize{}262.1}\tabularnewline
 & {\scriptsize{}400} & {\scriptsize{}329.7} & {\scriptsize{}330.2} & {\scriptsize{}330.1} & {\scriptsize{}330.2} & {\scriptsize{}330.2} & \textbf{\scriptsize{}91.8} & {\scriptsize{}262.1}\tabularnewline
 & {\scriptsize{}800} & {\scriptsize{}329.2} & {\scriptsize{}328.7} & {\scriptsize{}329.7} & {\scriptsize{}330.2} & {\scriptsize{}330.2} & \textbf{\scriptsize{}91.8} & {\scriptsize{}262.1}\tabularnewline
\hline 
\end{tabular}
\par\end{centering}
\caption{Minimum subgradient sizes for the FilmTrust dataset. Times are in
seconds and \textquotedblleft -\textquotedblright{} indicates a run
that did not generate a subgradient within the given time limit.}
\label{tabl:mc_filmtrust}
\end{table}

\subsubsection{Blockwise matrix completion}

Given a quadruple $(\alpha,\beta,\mu,\theta)\in\r_{++}^{4}$, a block
decomposable data matrix $A\in\r^{\ell\times n}$ with blocks $\{A_{i}\}_{i=1}^{k}\subseteq\r^{p\times q}$,
and indices $\Omega$, this subsection considers the following constrained
blockwise matrix completion (BMC) problem: 
\begin{align*}
\begin{aligned}\min_{U\in\r^{m\times n}}\  & \frac{1}{2}\|P_{\Omega}(U-A)\|_{F}^{2}+\sum_{i=1}^{k}\left[\kappa_{\mu}\circ\sigma(U_{i})+\tau_{\alpha}\circ\sigma(U_{i})\right]\\
\text{s.t.}\  & \|U\|_{F}^{2}\leq\sqrt{\ell n}\cdot\max_{i,j}|A_{ij}|,
\end{aligned}
\end{align*}
where $P_{\Omega}$, $\kappa_{\mu}$, and $\tau_{\alpha}$ are as
in Subsection~\ref{subsec:mc} and $U_{i}\in\r^{p\times q}$ is the
$i^{{\rm th}}$ block of $U$ with the same indices as $A_{i}$ with
respect to $A$.

We now describe the two classes of data matrices that are considered.
Every data matrix is a $5$-by-$5$ block matrix consisting of $50$-by-$100$
sized submatrices. Every submatrix contains only 25\% nonzero entries
and each data matrix generates its submatrix entries from different
probability distributions. More specifically, for a sampled probability
$p\sim\textsc{Uniform}[0,1]$ specific to a fixed submatrix, one class
uses a $\textsc{Binomial}(n,p)$ distribution with $n=10$, while
the other uses a $\textsc{TruncatedNormal}(\mu,\sigma)$ distribution
with $\mu=10p$, $\sigma^{2}=10p(1-p)$, and upper and lower bounds
$0$ and $10$, respectively.

We now describe the experiment parameters considered. First, the the
decomposition of the objective function and the quantities $Z_{0}$,
$(m_{1},M_{1})$, $(m_{2},M_{2})$, $\hat{\rho}$, and $\Omega$ are
the same as in Subsection~\ref{subsec:mc}. Second, we fix $(\alpha,\beta,\mu)=(10,20,2)$
and vary $(\theta,A)$ across the different problem instances. 

We now present the results. Figure~\ref{fig:bmc_tnormal} contains
the plots of the log objective function value against the runtime
for the binomial data set, listed in increasing order of $M_{2}$.
The corresponding plots for the truncated normal data set are similar
to the binomial plots so we omit them for the sake of brevity. Tables~\ref{tabl:bmc_binom}
and \ref{tabl:bmc_tnorm} present the minimal subgradient size obtained
within the time limit of 1000. Moreover, each row of these tables
corresponds to a different choice of $\theta$ and the bolded numbers
highlight which algorithm performed the best in terms of the size
obtained in a run.

\begin{figure}[tbhp]
\begin{centering}
\includesvg[width=1\linewidth]{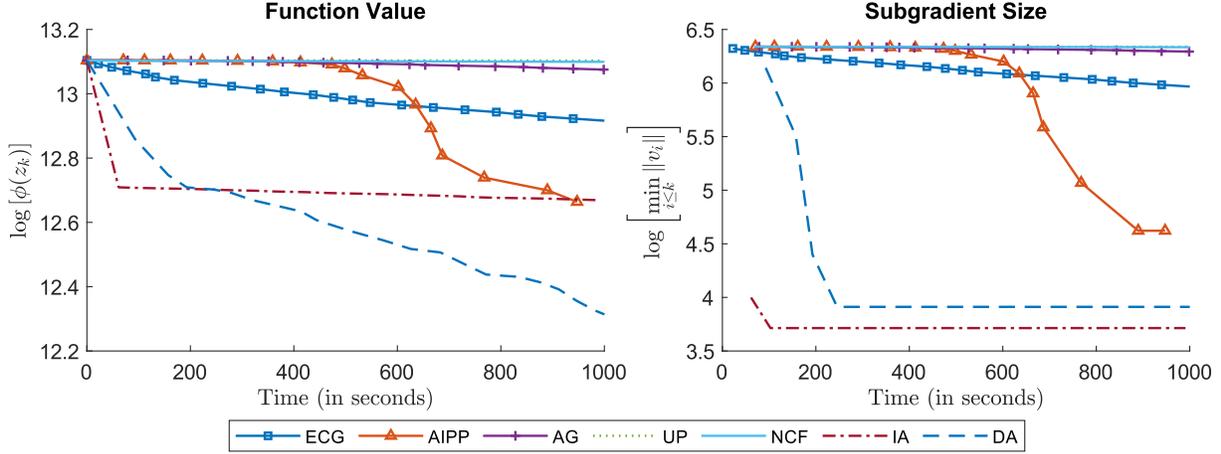}
\par\end{centering}
\caption{Function values and minimum subgradients for the truncated normal
dataset with $\theta=10^{-1}$.}

\centering{}\label{fig:bmc_tnormal} 
\end{figure}

\begin{table}[tbh]
\begin{centering}
\begin{tabular}{|>{\centering}m{1.8cm}|>{\centering}m{0.8cm}|>{\centering}p{1cm}>{\centering}p{1cm}>{\centering}p{1cm}>{\centering}p{1cm}>{\centering}p{1cm}>{\centering}p{1cm}>{\centering}p{1cm}|}
\hline 
{\footnotesize{}Parameters} & {\footnotesize{}Time} & \multicolumn{7}{c|}{{\footnotesize{}Minimum Subgradient Size ($\min_{i\leq k}\|\hat{V}_{i}\|$)}}\tabularnewline
\cline{3-9} \cline{4-9} \cline{5-9} \cline{6-9} \cline{7-9} \cline{8-9} \cline{9-9} 
{\footnotesize{}$(\theta,m,M)^{T}$} & {\footnotesize{}$t$} & {\footnotesize{}ECG} & {\footnotesize{}AIPP} & {\footnotesize{}AG} & {\footnotesize{}UP} & {\footnotesize{}NCF} & {\footnotesize{}IA} & {\footnotesize{}DA}\tabularnewline
\hline 
\multirow{4}{1.8cm}{{\scriptsize{}$\left[\begin{array}{c}
1\\
169\\
201
\end{array}\right]$}} & {\scriptsize{}100} & {\scriptsize{}392.4} & {\scriptsize{}500.3} & {\scriptsize{}501.2} & {\scriptsize{}506.0} & {\scriptsize{}482.5} & \textbf{\scriptsize{}33.9} & {\scriptsize{}75.5}\tabularnewline
 & {\scriptsize{}200} & {\scriptsize{}392.4} & {\scriptsize{}478.4} & {\scriptsize{}492.3} & {\scriptsize{}506.0} & {\scriptsize{}465.0} & \textbf{\scriptsize{}33.9} & {\scriptsize{}43.2}\tabularnewline
 & {\scriptsize{}400} & {\scriptsize{}392.4} & {\scriptsize{}182.2} & {\scriptsize{}455.9} & {\scriptsize{}57.1} & {\scriptsize{}407.0} & \textbf{\scriptsize{}33.9} & {\scriptsize{}43.2}\tabularnewline
 & {\scriptsize{}800} & {\scriptsize{}392.4} & {\scriptsize{}36.7} & {\scriptsize{}320.6} & {\scriptsize{}57.1} & {\scriptsize{}284.3} & \textbf{\scriptsize{}33.9} & {\scriptsize{}43.2}\tabularnewline
\hline 
\multirow{4}{1.8cm}{{\scriptsize{}$\left[\begin{array}{c}
0.1\\
11443\\
2001
\end{array}\right]$}} & {\scriptsize{}100} & {\scriptsize{}489.1} & {\scriptsize{}505.9} & {\scriptsize{}505.7} & {\scriptsize{}-} & {\scriptsize{}505.8} & \textbf{\scriptsize{}43.4} & {\scriptsize{}416.0}\tabularnewline
 & {\scriptsize{}200} & {\scriptsize{}476.9} & {\scriptsize{}505.6} & {\scriptsize{}505.3} & {\scriptsize{}505.5} & {\scriptsize{}505.5} & \textbf{\scriptsize{}43.4} & {\scriptsize{}76.9}\tabularnewline
 & {\scriptsize{}400} & {\scriptsize{}449.5} & {\scriptsize{}503.4} & {\scriptsize{}503.1} & {\scriptsize{}505.5} & {\scriptsize{}505.0} & \textbf{\scriptsize{}43.4} & {\scriptsize{}53.8}\tabularnewline
 & {\scriptsize{}800} & {\scriptsize{}399.4} & {\scriptsize{}240.8} & {\scriptsize{}496.2} & {\scriptsize{}505.3} & {\scriptsize{}503.9} & \textbf{\scriptsize{}43.4} & {\scriptsize{}53.8}\tabularnewline
\hline 
\multirow{4}{1.8cm}{{\scriptsize{}$\left[\begin{array}{c}
0.01\\
839400\\
20001
\end{array}\right]$}} & {\scriptsize{}100} & {\scriptsize{}505.6} & {\scriptsize{}505.9} & {\scriptsize{}505.8} & {\scriptsize{}505.9} & {\scriptsize{}505.9} & \textbf{\scriptsize{}48.6} & {\scriptsize{}137.5}\tabularnewline
 & {\scriptsize{}200} & {\scriptsize{}505.1} & {\scriptsize{}505.9} & {\scriptsize{}505.7} & {\scriptsize{}505.9} & {\scriptsize{}505.9} & \textbf{\scriptsize{}48.6} & {\scriptsize{}58.6}\tabularnewline
 & {\scriptsize{}400} & {\scriptsize{}504.1} & {\scriptsize{}498.1} & {\scriptsize{}504.9} & {\scriptsize{}505.9} & {\scriptsize{}505.9} & \textbf{\scriptsize{}48.6} & {\scriptsize{}58.6}\tabularnewline
 & {\scriptsize{}800} & {\scriptsize{}502.2} & {\scriptsize{}176.9} & {\scriptsize{}502.1} & {\scriptsize{}505.9} & {\scriptsize{}505.9} & \textbf{\scriptsize{}48.6} & {\scriptsize{}58.6}\tabularnewline
\hline 
\end{tabular}
\par\end{centering}
\caption{Minimum subgradient sizes for the binomial dataset. Times are in seconds
and \textquotedblleft -\textquotedblright{} indicates a run that did
not generate a subgradient within the given time limit.}

\label{tabl:bmc_binom}
\end{table}

\begin{table}[tbh]
\begin{centering}
\begin{tabular}{|>{\centering}m{1.8cm}|>{\centering}m{0.8cm}|>{\centering}p{1cm}>{\centering}p{1cm}>{\centering}p{1cm}>{\centering}p{1cm}>{\centering}p{1cm}>{\centering}p{1cm}>{\centering}p{1cm}|}
\hline 
{\footnotesize{}Parameters} & {\footnotesize{}Time} & \multicolumn{7}{c|}{{\footnotesize{}Minimum Subgradient Size ($\min_{i\leq k}\|\hat{V}_{i}\|$)}}\tabularnewline
\cline{3-9} \cline{4-9} \cline{5-9} \cline{6-9} \cline{7-9} \cline{8-9} \cline{9-9} 
{\footnotesize{}$(\theta,m,M)^{T}$} & {\footnotesize{}$t$} & {\footnotesize{}ECG} & {\footnotesize{}AIPP} & {\footnotesize{}AG} & {\footnotesize{}UP} & {\footnotesize{}NCF} & {\footnotesize{}IA} & {\footnotesize{}DA}\tabularnewline
\hline 
\multirow{4}{1.8cm}{{\footnotesize{}$\left[\begin{array}{c}
1\\
169\\
201
\end{array}\right]$}} & {\footnotesize{}100} & {\scriptsize{}-} & {\scriptsize{}564.3} & {\scriptsize{}562.7} & {\scriptsize{}-} & {\scriptsize{}552.2} & \textbf{\scriptsize{}39.1} & {\scriptsize{}362.3}\tabularnewline
 & {\footnotesize{}200} & {\scriptsize{}433.5} & {\scriptsize{}551.8} & {\scriptsize{}554.1} & {\scriptsize{}566.6} & {\scriptsize{}536.2} & \textbf{\scriptsize{}30.0} & {\scriptsize{}80.3}\tabularnewline
 & {\footnotesize{}400} & {\scriptsize{}433.5} & {\scriptsize{}351.5} & {\scriptsize{}526.6} & {\scriptsize{}566.6} & {\scriptsize{}501.7} & \textbf{\scriptsize{}30.0} & {\scriptsize{}40.8}\tabularnewline
 & {\footnotesize{}800} & {\scriptsize{}433.5} & {\scriptsize{}35.6} & {\scriptsize{}433.7} & {\scriptsize{}55.8} & {\scriptsize{}435.7} & \textbf{\scriptsize{}30.0} & {\scriptsize{}40.8}\tabularnewline
\hline 
\multirow{4}{1.8cm}{{\footnotesize{}$\left[\begin{array}{c}
0.1\\
11443\\
2001
\end{array}\right]$}} & {\footnotesize{}100} & {\scriptsize{}533.8} & {\scriptsize{}566.4} & {\scriptsize{}566.2} & {\scriptsize{}-} & {\scriptsize{}566.2} & \textbf{\scriptsize{}41.0} & {\scriptsize{}465.0}\tabularnewline
 & {\footnotesize{}200} & {\scriptsize{}507.4} & {\scriptsize{}566.1} & {\scriptsize{}565.7} & {\scriptsize{}566.0} & {\scriptsize{}566.0} & \textbf{\scriptsize{}41.0} & {\scriptsize{}81.4}\tabularnewline
 & {\footnotesize{}400} & {\scriptsize{}478.2} & {\scriptsize{}563.6} & {\scriptsize{}561.8} & {\scriptsize{}566.0} & {\scriptsize{}565.6} & \textbf{\scriptsize{}41.0} & {\scriptsize{}50.0}\tabularnewline
 & {\footnotesize{}800} & {\scriptsize{}417.6} & {\scriptsize{}159.0} & {\scriptsize{}549.9} & {\scriptsize{}565.8} & {\scriptsize{}564.4} & \textbf{\scriptsize{}41.0} & {\scriptsize{}50.0}\tabularnewline
\hline 
\multirow{4}{1.8cm}{{\footnotesize{}$\left[\begin{array}{c}
0.01\\
839400\\
20001
\end{array}\right]$}} & {\footnotesize{}100} & {\scriptsize{}565.5} & {\scriptsize{}566.4} & {\scriptsize{}566.2} & {\scriptsize{}566.4} & {\scriptsize{}566.4} & \textbf{\scriptsize{}45.8} & {\scriptsize{}54.3}\tabularnewline
 & {\footnotesize{}200} & {\scriptsize{}564.6} & {\scriptsize{}563.9} & {\scriptsize{}565.5} & {\scriptsize{}566.4} & {\scriptsize{}566.4} & \textbf{\scriptsize{}45.8} & {\scriptsize{}54.3}\tabularnewline
 & {\footnotesize{}400} & {\scriptsize{}562.7} & {\scriptsize{}186.1} & {\scriptsize{}563.1} & {\scriptsize{}566.3} & {\scriptsize{}566.4} & \textbf{\scriptsize{}45.8} & {\scriptsize{}54.3}\tabularnewline
 & {\footnotesize{}800} & {\scriptsize{}559.1} & {\scriptsize{}143.6} & {\scriptsize{}555.6} & {\scriptsize{}566.3} & {\scriptsize{}566.3} & \textbf{\scriptsize{}45.8} & {\scriptsize{}54.3}\tabularnewline
\hline 
\end{tabular}
\par\end{centering}
\caption{Minimum subgradient sizes for the truncated normal dataset. Times
are in seconds and \textquotedblleft -\textquotedblright{} indicates
a run that did not generate a subgradient within the given time limit.}

\label{tabl:bmc_tnorm}
\end{table}

\subsection{General Comments}

This subsection makes two comments about the results obtained in the previous subsection.
First, within the alloted time (i.e., 1000 seconds),
the DA-ICG and IA-ICG methods obtained approximate solutions with small primal residual $\|\hat{V}_{k}\|$
much faster than  the other
first-order methods. More specifically, the former methods were able to obtain higher quality solutions much sooner
than the latter ones, i.e, within the first 100 seconds.
Second, the larger the ratio $m/M$
is, the more efficient the ICG methods are compared to the other benchmarked
methods.

\section{Static ICG Iteration Complexities}

\label{sec:icg_cvg_rate}

This section establishes the iteration complexities for each of the
static ICG methods in Section~\ref{sec:accelerated_icg}.

\subsection{\label{subsec:aicg}Static IA-ICG Iteration Complexity}

This subsection establishes the key properties of the static IA-ICG
method. 
\begin{lem}
\label{lem:aicg_resid_rate}Let $\{(\icgY i,\hat{\icgY{}}_{i},\hat{v}_{i})\}_{i=1}^{k}$
be the collection of iterates generated by the static IA-ICG method.
For every $i\geq1$, we have 
\begin{align}
\frac{1}{4\lam}\|\icgY{i-1}-\icgY i\|^{2} & \leq\phi(\icgY{i-1})-\widetilde{\ell}_{\phi}(\icgY i;\icgY{i-1})-\frac{M_{1}}{2}\|\icgY i-\icgY{i-1}\|^{2}\leq\phi(\icgY{i-1})-\phi(\icgY i),\label{eq:aicg_descent}
\end{align}
where $\widetilde{\ell}_{\phi}$ is as in \eqref{eq:ell_phi_C_bar_lam_def}. 
\end{lem}

\begin{proof}
Let $i\geq1$ be fixed and let $(\icgY i,v_{i},\varepsilon_{i})$
be the point output by the $i^{{\rm th}}$ successful call to the
R-ACG algorithm. Moreover, let $\Delta_{1}(\cdot;\cdot,\cdot)$ be
as in \eqref{eq:Delta_def} with $(\psi_{s},\psi_{n})$ given by \eqref{eq:gen_acg_inputs}.
Using the definition of $\widetilde{\ell}_{\phi}$, step~2 of the
method, and fact that $(\aicgY{},v,\varepsilon)=(\icgY i,v_{i},\varepsilon_{i})$
solves Problem~${\cal B}$ in Section~\ref{sec:background} with
$(\mu,\psi_{s},\psi_{n})$ as in \eqref{eq:gen_acg_inputs}, we have
that 
\[
\varepsilon_{i}\geq\Delta_{1}(\icgY{i-1};\icgY i,v_{i})=\lam\widetilde{\ell}_{\phi}(\icgY i;\icgY{i-1})-\lam\phi(\icgY{i-1})-\inner{v_{i}}{\icgY i-\icgY{i-1}}+\|\icgY i-\icgY{i-1}\|^{2}.
\]
Rearranging the above inequality and using assumption (A2), \eqref{eq:aicg_lam_restr},
and the fact that $\left\langle a,b\right\rangle \geq-\|a\|^{2}/2-\|b\|^{2}/2$
for every $a,b\in{\cal Z}$ yields 
\begin{align}
 & \lam\phi(\icgY{i-1})-\lam\widetilde{\ell}_{\phi}(\icgY i;\icgY{i-1})\geq\left\langle v_{i},\icgY{i-1}-\icgY i\right\rangle -\varepsilon_{i}+\|\icgY i-\icgY{i-1}\|^{2}\nonumber \\
 & =\frac{1}{2}\|\icgY i-\icgY{i-1}\|^{2}-\frac{1}{2}\left(\|v_{i}\|^{2}+2\varepsilon_{i}\right)\geq\left(\frac{1-\theta^{2}}{2}\right)\|\icgY i-\icgY{i-1}\|^{2}\nonumber \\
 & =\frac{\lam M_{1}}{2}\|\icgY i-\icgY{i-1}\|^{2}+\left(\frac{1-\lam M_{1}-\theta^{2}}{2}\right)\|\icgY i-\icgY{i-1}\|^{2}\nonumber \\
 & =\frac{\lam M_{1}}{2}\|\icgY i-\icgY{i-1}\|^{2}+\frac{1}{4}\|\icgY i-\icgY{i-1}\|^{2}.\label{eq:main_aicg_resid_bd1}
\end{align}
Rearranging terms yields the first inequality of \eqref{eq:aicg_descent}.
The second inequality of \eqref{eq:aicg_descent} follows from the
first inequality, the fact that $\widetilde{\ell}_{\phi}(\icgY i;\icgY{i-1})+M_{1}\|\icgY i-\icgY{i-1}\|^{2}/2\geq\phi(\icgY i)$
from assumption (A2), and the definition of $\widetilde{\ell}_{\phi}$. 
\end{proof}
The next results establish the rate at which the residual $\|\hat{v}_{i}\|$
tends to 0. 
\begin{lem}
\label{lem:p_norm_tech}Let $p>1$ be given. Then, for every $a,b\in\r^{k}$,
we have 
\[
\min_{1\leq i\leq k}\left\{ |a_{i}b_{i}|\right\} \leq k^{-p}\left\Vert a\right\Vert _{1}\left\Vert b\right\Vert _{1/(p-1)}.
\]
\end{lem}

\begin{proof}
Let $p>1$ and $a,b\in\r^{k}$ be fixed and let $q\geq1$ be such
that $p^{-1}+q^{-1}=1$. Using the fact that $\left\langle x,y\right\rangle \leq\|x\|_{p}\|y\|_{q}$
for every $x,y\in\r^{k}$, and denoting $\tilde{a}$ and $\tilde{b}$
to be vectors with entries $|a_{i}|^{1/p}$ and $|b_{i}|^{1/p}$,
respectively, we have that 
\begin{align*}
 & k\min_{1\leq i\leq k}\left\{ |a_{i}b_{i}|\right\} ^{1/p}\leq\sum_{i=1}^{k}|a_{i}b_{i}|^{1/p}\\
 & \leq\|\tilde{a}\|_{p}\|\tilde{b}\|_{q}=\|a\|_{1}^{1/p}\left(\sum_{i=1}^{k}|b_{i}|^{q/p}\right)^{1/q}=\left(\|a\|_{1}\|b\|_{q/p}\right)^{1/p}.
\end{align*}
Dividing by $k$, taking the $p^{{\rm th}}$ power on both sides,
and using the fact that $p/q=p-1$, yields 
\[
\min_{1\leq i\leq k}\left\{ |a_{i}b_{i}|\right\} \leq k^{-p}\|a\|_{1}\|b\|_{q/p}=k^{-p}\|a\|_{1}\|b\|_{1/(p-1)}.
\]
\end{proof}
\begin{prop}
\label{prop:aicg_v_hat_rate_alt}Let $\{(\icgY i,\hat{\icgY{}}_{i},\hat{v}_{i})\}_{i=1}^{k}$
be as in Lemma~\ref{lem:aicg_resid_rate} and define the quantities
\begin{gather}
\begin{gathered}L_{1,k}^{{\rm avg}}:=\frac{1}{k}\sum_{i=1}^{k}L_{1}(\icgY i,\icgY{i-1}),\quad C_{\lam,k}^{{\rm avg}}:=\frac{1}{k}\sum_{i=1}^{k}C_{\lam}(\hat{\icgY{}}_{i},\icgY i),\\
D_{k}^{{\rm avg}}:=L_{1,k}^{{\rm avg}}+\frac{\theta}{\lam}C_{\lam,k}^{{\rm avg}},\quad\beta_{1}:=\left(\frac{1+\overline{C}_{\lam}}{\lam}\right)+\sqrt{2}\left(\frac{2+\lam L_{1}+\theta\overline{C}_{\lam}}{\lam}\right),
\end{gathered}
\label{eq:avg_def}
\end{gather}
where $C_{\lam}(\cdot,\cdot)$ and $\overline{C}_{\lam}$ are as in
\eqref{eq:C_lam_fn_def} and \eqref{eq:ell_phi_C_bar_lam_def}, respectively.
Then, we have 
\[
\min_{i\leq k}\|\hat{v}_{i}\|={\cal O}_{1}\left(\left[\sqrt{\lam}L_{1,k}^{{\rm avg}}+\frac{1+\theta C_{\lam,k}^{{\rm avg}}}{\sqrt{\lam}}\right]\left[\frac{\phi(z_{0})-\phi_{*}}{k}\right]^{1/2}\right)+\frac{\hat{\rho}}{2}.
\]
\end{prop}

\begin{proof}
Using Lemma~\ref{lem:spec_refine} with $(\icgY{},w)=(\icgY i,\icgY{i-1})$
and the fact that $C_{\lam}(\cdot,\cdot)\leq\overline{C}_{\lam}$
and $L_{1}(\cdot,\cdot)\leq L_{1}$, we have $\|\hat{v}_{i}\|\le{\cal E}_{i}\|\icgY i-\icgY{i-1}\|$,
for every $i\leq k$, where 
\[
{\cal E}_{i}:=\frac{2+\lam L_{1}(\icgY i,\icgY{i-1})+\theta C_{\lam}(\hat{\icgY{}}_{i},\icgY i)}{\lam}\quad\forall i\geq1.
\]
As a consequence, using the sum of the second bound in Lemma~\ref{lem:aicg_resid_rate}
from $i=1$ to $k$, the definitions in \eqref{eq:avg_def}, and Lemma~\ref{lem:p_norm_tech}
with $p=3/2$, $a_{i}={\cal E}_{i}$, and $b_{i}=\|\icgY i-\icgY{i-1}\|$
for $i=1$ to $k$, yields 
\begin{align}
 & \min_{i\leq k}\|\hat{v}_{i}\|\leq\min_{i\leq k}{\cal E}_{i}\|\icgY i-\icgY{i-1}\|\leq\frac{1}{k^{3/2}}\left(\sum_{i=1}^{k}{\cal E}_{i}\right)\left(\sum_{i=1}^{k}\|\icgY i-\icgY{i-1}\|^{2}\right)^{1/2}\nonumber \\
 & ={\cal O}_{1}\left(\left[\sqrt{\lam}L_{1,k}^{{\rm avg}}+\frac{1+\theta C_{\lam,k}^{{\rm avg}}}{\sqrt{\lam}}\right]\left[\frac{\phi(z_{0})-\phi_{*}}{k}\right]^{1/2}\right).\label{eq:v_hat_aicg_bd2}
\end{align}
\end{proof}
We are now ready to give the proof of Theorem~\ref{thm:aicg_compl}. 
\begin{proof}[\textit{Proof of Theorem~\ref{thm:aicg_compl}}]
(a) This follows from Proposition~\ref{prop:aicg_v_hat_rate_alt},
the fact that $C_{\lam}(\cdot,\cdot)\leq\overline{C}_{\lam}$ and
$L_{f_{1}}(\cdot,\cdot)\leq L_{1}$, and the stopping condition in
step~3.

(b) The fact that $(\hat{\icgY{}},\hat{v})=(\hat{\icgY{}}_{k},\hat{v}_{k})$
satisfies the inclusion of \eqref{eq:rho_approx_soln} follows from
Lemma~\ref{lem:spec_refine} with $(\icgY{},v,w)=(\icgY k,v_{k},\icgY{k-1})$.
The fact that $\|\hat{v}\|\leq\hat{\rho}$ follows from the stopping
condition in step~3.

(c) This follows from Proposition~\ref{prop:acg_properties}(c) and
the fact that method stops in finite number of iterations from part
(a). 
\end{proof}

\subsection{Static DA-ICG Iteration Complexity}

This subsection establishes several key properties of static DA-ICG
method.

To avoid repetition, we assume throughout this subsection that $k\geq1$
denotes an arbitrary successful outer iteration of the DA-ICG method
and let 
\[
\{(a_{i},A_{i},\aicgYMin i,\aicgY i,\aicgX i,\aicgXTilde{i-1},\hat{\icgY{}}_{i},\hat{v}_{i},v_{i},\varepsilon_{i})\}_{i=1}^{k}
\]
denote the sequence of all iterates generated by it up to and including
the $k^{{\rm th}}$ iteration. Observe that this implies that the
$i^{{\rm th}}$ DA-ICG outer iteration for any $1\leq i\leq k$ is
successful, i.e., the (only) R-ACG call in step~2 of the DA-ICG method
does not stop with failure and $\Delta_{1}(\icgY{i-1};\aicgY i,v_{i})\leq\varepsilon_{i}$.
Moreover, throughout this subsection we let 
\begin{equation}
\gammaBTFn_{i}(u)=\ell_{f_{1}}(u;\aicgXTilde{i-1})+f_{2}(u)+h(u),\quad\gammaBFn_{i}(u)=\gammaBTFn_{i}(\aicgY i)+\frac{1}{\lam}\inner{v_{i}+\aicgXTilde{i-1}-\aicgY i}{u-\aicgY i}.\label{eq:theory_gamma}
\end{equation}

The first set of results present some basic properties about the functions
$\gammaBTFn_{i}$ and $\gammaBFn_{i}$ as well as the iterates generated
by the method.
\begin{lem}
\label{lem:gamma_props}Let $\Delta_{1}(\cdot;\cdot,\cdot)$ be as
in \eqref{eq:Delta_def} with $(\psi_{s},\psi_{n})$ given by \eqref{eq:gen_acg_inputs}.
Then, the following statements hold for any $s\in\dom h$ and $1\leq i\leq k$: 
\begin{itemize}
\item[(a)] $\gammaBFn_{i}(\aicgY i)=\gammaBTFn_{i}(\aicgY i)$; 
\item[(b)] $\aicgX i=\argmin_{u\in\Omega}\left\{ \lam a_{i-1}\gammaBFn_{i}(u)+\|u-\aicgX{i-1}\|^{2}/2\right\} ;$ 
\item[(c)] $\aicgY i-v_{i}=\argmin_{u\in{\cal Z}}\left\{ \lam\gammaBFn_{i}(u)+\|u-\aicgXTilde{i-1}\|^{2}/2\right\} ;$ 
\item[(d)] $-M_{1}\|u-\aicgXTilde{i-1}\|^{2}/2\leq\gammaBTFn_{i}(u)-\phi(u)\leq m_{1}\|u-\aicgXTilde{i-1}\|^{2}/2$; 
\item[(e)] $\phi(\aicgYMin{i-1})\geq\phi(\aicgYMin i)$ and $\phi(\aicgY i)\geq\phi(\aicgYMin i)$. 
\end{itemize}
\end{lem}

\begin{proof}
To keep the notation simple, denote 
\begin{gather}
\begin{gathered}(\YP,\YMinP,\YM,\XtM)=(\aicgY i,\aicgYMin i,\aicgYMin{i-1},\aicgXTilde{i-1}),\quad(\XP,\XM)=(\aicgX i,\aicgX{i-1}),\\
(\AP,\AM,\aM)=(A_{i},A_{i-1},a_{i-1}),\quad(v,\varepsilon)=(v_{i},\varepsilon_{i}).
\end{gathered}
\label{eq:d_aicg_proof_notation}
\end{gather}

(a) This is immediate from the definitions of $\gammaBFn$ and $\gammaBTFn$
in \eqref{eq:theory_gamma}.

(b) Define $\aicgXhat i:=\aicgX{k-1}-a_{k-1}\left(v_{k}+\aicgXTilde{k-1}-\aicgY k\right)$.
Using the definition of $\gammaBFn$ in \eqref{eq:theory_gamma},
we have that 
\begin{align*}
 & \argmin_{u\in\Omega}\left\{ \lam\aM\gammaFn u+\frac{1}{2}\|u-\XM\|^{2}\right\} =\argmin_{u\in\Omega}\left\{ a\left\langle v+\XtM-\YP,u-x\right\rangle +\frac{1}{2}\|u-\XM\|^{2}\right\} \\
 & =\argmin_{u\in\Omega}\frac{1}{2}\left\Vert u-\left(\XM-a\left[v+\XtM-\YP\right]\right)\right\Vert ^{2}=\argmin_{u\in\Omega}\frac{1}{2}\left\Vert u-\XhP\right\Vert ^{2}=\XP.
\end{align*}

(c) Using the definition of $\gammaBFn$ in \eqref{eq:theory_gamma},
we have that 
\[
\lam\nabla\gammaFn{\YP-v}+(\YP-v)-\XtM=(v+\XtM-\YP)+(\YP-v)-\XtM=0,
\]
and hence, the point $\YP-v$ is the global minimum of $\lam\gammaBFn+\|\cdot-\XtM\|^{2}/2$.

(d) This follows from inequality \eqref{eq:curv_fi} with $i=1$ and
the definition of $\gammaBTFn$ in \eqref{eq:theory_gamma}.

(e) This follows immediately from the update rule of $\aicgYMin i$
in \eqref{eq:z_c_def}. 
\end{proof}
\begin{lem}
\label{lem:d_aicg_Delta_props} Let $w=\aicgXTilde{i-1}$, the pair
$(\psi_{n},\psi_{s})$ be as in \eqref{eq:gen_acg_inputs}, and $\Delta_{1}(\cdot;\cdot,\cdot)$
be as in \eqref{eq:Delta_def} with $(\psi_{s},\psi_{n})$ given by
\eqref{eq:gen_acg_inputs}. Then, following statements hold: 
\begin{itemize}
\item[(a)] the triple $(\aicgY i,v_{i},\varepsilon_{i})$ solves Problem~${\cal B}$
and satisfies $\Delta_{1}(\icgY{i-1};\aicgY i,v_{i})\leq\varepsilon$,
and hence 
\begin{gather}
\begin{gathered}\|v_{i}\|+2\varepsilon_{i}\leq\theta^{2}\|\aicgY i-\aicgXTilde{i-1}\|^{2},\quad\Delta_{1}(u;\aicgY i,v_{i})\leq\varepsilon_{i}\quad\forall u\in\{\hat{\icgY{}}_{i},\aicgYMin{i-1}\},\end{gathered}
\label{eq:d_aicg_main_ineq}
\end{gather}
\item[(b)] if $f_{2}$ is convex, then $(\aicgY i,v_{i},\varepsilon_{i})$ solves
Problem~${\cal A}$; 
\item[(c)] $\Delta_{1}(s;\aicgY i,v_{i})=\lam[\gammaBFn_{i}(s)-\gammaBTFn_{i}(s)];$ 
\item[(d)] $\Delta_{1}(\aicgYMin i;\aicgY i,v_{i})\leq\varepsilon$. 
\end{itemize}
\end{lem}

\begin{proof}
(a) This follows from step~2 of the DA-ICG method and Proposition~\ref{prop:acg_properties}(b).

(b) This follows from steps~2 and 3 of the DA-ICG method, the fact
that $h$ is convex, and Proposition~\ref{prop:acg_properties}(c)
with $\psi_{s}=\gammaBTFn_{i}+\|\cdot-\aicgXTilde{i-1}\|^{2}/2$.

(c) Using the definitions of $(\psi_{s},\psi_{n})$ and $(\gammaBFn,\gammaBTFn)$
in \eqref{eq:gen_acg_inputs} and \eqref{eq:theory_gamma}, respectively,
we have that 
\begin{align*}
 & \Delta_{1}(s;\YP,v)=(\psi_{s}+\psi_{n})(\YP)-(\psi_{s}+\psi_{n})(s)-\left\langle v,\YP-s\right\rangle +\frac{1}{2}\|s-\YP\|^{2}\\
 & =\left[\lam\widetilde{\gamma}(\YP)+\frac{1}{2}\|\YP-\tilde{x}\|^{2}\right]-\left[\lam\widetilde{\gamma}(s)+\frac{1}{2}\|s-\tilde{x}\|^{2}\right]-\left\langle v,\YP-s\right\rangle +\frac{1}{2}\|s-\YP\|^{2}\\
 & =\left[\lam\gamma(s)+\frac{1}{2}\|s-\tilde{x}\|^{2}\right]-\left[\lam\widetilde{\gamma}(s)+\frac{1}{2}\|s-\tilde{x}\|^{2}\right]=\lam\gamma(s)-\lam\widetilde{\gamma}(s).
\end{align*}

(d) If $\aicgYMin i=\aicgYMin{i-1}$, then this follows from step~3
of the method. On the other hand, if $\aicgYMin i=\aicgY i$, then
this follows from part (c). 
\end{proof}
We now state (without proof) some well-known properties of $A_{i}$
and $a_{i-1}$. 
\begin{lem}
\label{lem:A_k_props} For every $1\leq i\leq k$, we have that: 
\begin{itemize}
\item[(a)] $a_{i-1}^{2}=A_{i}$; 
\item[(b)] $i^{2}/4\leq A_{i}\leq i^{2}$. 
\end{itemize}
\end{lem}

The next two lemmas are technical results that are needed to establish
the key inequality in Proposition~\ref{prop:descent_d_aicg}. 
\begin{lem}
\label{lem:main_resid_bd} For every $u\in\dom h$ and $1\leq i\leq k$,
we have that 
\begin{gather*}
\frac{1}{2}\left(A_{i-1}\|\aicgYMin{i-1}-\aicgXTilde{i-1}\|^{2}+a_{i-1}\|u-\aicgXTilde{i-1}\|^{2}\right)\leq2D_{\Omega}^{2}+a_{i-1}D_{h}^{2}.
\end{gather*}
\end{lem}


\begin{proof}
Throughout the proof, we use the notation in \eqref{eq:d_aicg_proof_notation}.
Using the relation $(p+q)^{2}\leq2p^{2}+2q^{2}$ for every $p,q\in\r$,
Lemma~\ref{lem:A_k_props}(a), the fact that $A\leq A^{+}$, $x\in\Omega$,
and $y\in\dom h$, and the definitions of $\XtM$ in \eqref{eq:accel_d_aicg_def}
and of $D_{\Omega}$ and $D_{h}$ in \eqref{eq:d_aicg_diam}, we conclude
that 
\begin{align*}
 & \AM\|\YM-\XtM\|^{2}+\aM\|u-\XtM\|^{2}=\AM\left\Vert \frac{\aM}{\AP}(\YM-\XM)\right\Vert ^{2}+\aM\left\Vert \frac{\AM}{\AP}(u-\YM)+\frac{\aM}{\AP}(u-\XM)\right\Vert ^{2}\\
 & \leq\frac{A}{\AP}\left(\|(\YM-u)+(u-\XM)\|^{2}+2a\left[\frac{\AM^{2}}{\AP^{2}}\|u-\YM\|^{2}+\frac{a^{2}}{\AP^{2}}\|u-\XM\|^{2}\right]\right)\\
 & \leq\frac{2A}{A^{+}}\left(\|u-\YM\|^{2}+\|u-\XM\|^{2}\right)+2\aM\|u-\YM\|^{2}+\frac{2\aM}{\AP}\|u-\XM\|^{2}\\
 & \leq2\left[\|u-x\|^{2}+(1+\aM)\|u-y\|^{2}\right]\leq2[D_{\Omega}^{2}+(1+\aM)D_{h}^{2}].
\end{align*}
The conclusion now follows from dividing both sides of the above inequalities
by 2 and using the fact that $D_{h}\leq D_{\Omega}$. 
\end{proof}

\begin{lem}
For every $u\in\dom h$ and $1\leq i\leq k$, we have that 
\begin{align}
 & A_{i}\left[\phi(\aicgYMin i)+\left(\frac{1-\lambda M_{1}}{2\lambda}\right)\|\aicgY i-\aicgXTilde{i-1}\|^{2}-\frac{\|v_{i}\|^{2}}{2\lam}\right]+\frac{1}{2\lambda}\|u-\aicgX i\|^{2}\nonumber \\
 & \leq A_{i-1}\gammaBFn_{i}(\aicgYMin{i-1})+a_{i-1}\gammaBFn_{i}(u)+\frac{1}{2\lambda}\|u-\aicgX{i-1}\|^{2}.\label{eq:d_aicg_subdescent}
\end{align}
\end{lem}

\begin{proof}
Throughout the proof, we use the notation in \eqref{eq:d_aicg_proof_notation}.
We first present two key expressions. First, using the definition
of $\gammaBFn$ in \eqref{eq:theory_gamma} and Lemma~\ref{lem:gamma_props}(c),
it follows that 
\begin{align}
 & \min_{u\in{\cal Z}}\left\{ \lam\gammaFn u+\frac{1}{2}\|u-\XtM\|^{2}\right\} =\lam\widetilde{\gamma}(\YP)-\left\langle v+\XtM-\YP,v\right\rangle +\frac{1}{2}\left\Vert v+\XtM-\YP\right\Vert ^{2}\nonumber \\
 & =\lam\widetilde{\gamma}(\YP)-\|v\|^{2}-\left\langle v,\XtM-\YP\right\rangle +\frac{1}{2}\left\Vert v+\XtM-\YP\right\Vert ^{2}\nonumber \\
 & =\lam\widetilde{\gamma}(\YP)-\frac{1}{2}\|v\|^{2}+\frac{1}{2}\|\XtM-\YP\|^{2}.\label{eq:gamma_reg_t_bd}
\end{align}
Second, Lemma~\ref{lem:gamma_props}(b) and the fact that the function
$\aM\gammaBFn+\|\cdot-\XM\|^{2}/(2\lam)$ is $(1/\lam)$--strongly
convex imply that 
\begin{equation}
\aM\gammaFn{\XP}+\frac{1}{2\lam}\|\XP-\XM\|^{2}\leq\aM\gammaFn u+\frac{1}{2\lam}\|u-\XM\|^{2}-\frac{1}{2\lam}\|u-\XP\|^{2}.\label{eq:XP_strong_cvx}
\end{equation}
Using \eqref{eq:gamma_reg_t_bd}, Lemma~\ref{lem:gamma_props}(d)--(e),
Lemma~\ref{lem:A_k_props}(a), and the fact that $\gammaBFn$ is
affine, we have that 
\begin{align}
 & \AP\left[\phi(\YMinP)+\left(\frac{1-\lambda M_{1}}{2\lambda}\right)\|\YP-\XtM\|^{2}\right]\leq\AP\left[\gammaTFn{\YP}+\frac{1}{2\lam}\|\YP-\XtM\|^{2}\right]\nonumber \\
 & =\AP\left[\min_{u\in{\cal Z}}\left\{ \gammaFn u+\frac{1}{2\lam}\|u-\XtM\|^{2}\right\} +\frac{\|v\|^{2}}{2\lam}\right]\nonumber \\
 & \leq\AP\left[\gammaFn{\frac{\AM\YM+\aM\XP}{\AP}}+\frac{1}{2\lambda}\left\Vert \frac{\AM\YM+\aM\XP}{\AP}-\frac{\AM\YM+\aM\XM}{\AP}\right\Vert ^{2}+\frac{\|v\|^{2}}{2\lam}\right]\nonumber \\
 & =\AM\gammaFn{\YM}+\aM\gammaFn{\XP}+\frac{\aM^{2}}{2\lambda\AP}\|\XM-\XP\|^{2}+\frac{\AP}{2\lam}\|v\|^{2}\nonumber \\
 & =\AM\gammaFn{\YM}+\aM\gammaFn{\XP}+\frac{1}{2\lambda}\|\XM-\XP\|^{2}+\frac{\AP}{2\lam}\|v\|^{2}\label{eq:sub_descent_ineq}
\end{align}
The conclusion now follows from combining \eqref{eq:XP_strong_cvx}
with \eqref{eq:sub_descent_ineq}. 
\end{proof}
We now present an inequality that plays an important role in the analysis
of the DA-ICG method. 
\begin{prop}
\label{prop:descent_d_aicg} Let $\Delta_{1}(\cdot;\cdot,\cdot)$
be as in \eqref{eq:Delta_def} with $(\psi_{s},\psi_{n})$ as in \eqref{eq:gen_acg_inputs},
and define 
\begin{equation}
\theta_{i}(u):=A_{i}\left[\phi(\aicgYMin i)-\phi(u)\right]+\frac{1}{2\lam}\|u-\aicgX i\|^{2}\quad\forall i\geq0.\label{eq:theta_def}
\end{equation}
For every $u\in\dom h$ satisfying $\Delta_{1}(u;\aicgY i,v_{i})\leq\varepsilon$
and $1\leq i\leq k$, we have that 
\begin{equation}
\frac{A_{i}}{4\lam}\|\aicgY i-\aicgXTilde{i-1}\|^{2}\leq m_{1}^{+}\left(a_{i-1}D_{h}^{2}+2D_{\Omega}^{2}\right)+\theta_{i-1}(u)-\theta_{i}(u).\label{eq:descent_d_aicg}
\end{equation}
\end{prop}

\begin{proof}
Throughout the proof, we use the notation in \eqref{eq:d_aicg_proof_notation}
together with the notation $\thetaM=\thetaM_{i-1}$ and $\thetaP=\thetaM_{i}$.
Let $u\in\dom h$ be such that $\Delta_{1}(u;\YP,v)\leq\varepsilon$.
Subtracting $A\phi(u)$ from both sides of the inequality in \eqref{eq:d_aicg_subdescent}
and using the definition of $\thetaP$ we have 
\begin{align}
 & \frac{\AP}{2\lam}\left[(1-\lam M_{1})\|\YP-\XtM\|^{2}-\|v\|^{2}\right]+\thetaP(u)\nonumber \\
 & =\frac{\AP}{2\lam}\left[(1-\lam M_{1})\|\YP-\XtM\|^{2}-\|v\|^{2}\right]+\AP\left[\phi(\YMinP)-\phi(u)\right]+\frac{1}{2\lam}\|u-\YP\|^{2}\nonumber \\
 & \leq\AM\gammaFn{\YM}+\aM\gammaFn u-\AM\phi(u)+\frac{1}{2\lambda}\|u-\XM\|^{2}\nonumber \\
 & =\aM\left[\gammaFn u-\phi(u)\right]+\AM\left[\gammaFn{\YM}-\phi(\YM)\right]+\thetaM(u).\label{eq:descent2_ineq1}
\end{align}
Moreover, using Lemma~\ref{lem:d_aicg_Delta_props}(a) and (c), and
with our assumption that $\Delta_{1}(u;\YP,v)\leq\varepsilon$, we
have that 
\begin{align}
\gammaFn s-\phi(s) & =\gammaTFn s-\phi(s)+\frac{\Delta_{1}(s;\YP,v)}{\lam}\leq\frac{m_{1}^{+}}{2}\|s-\XtM\|^{2}+\frac{\varepsilon}{\lam}\quad\forall s\in\{u,\YM\}.\label{eq:descent2_ineq2}
\end{align}
Combining \eqref{eq:descent2_ineq1}, \eqref{eq:descent2_ineq2},
and Lemma~\ref{lem:main_resid_bd} then yields 
\begin{align*}
 & \frac{\AP}{2\lam}\left[(1-\lam M_{1})\|\YP-\XtM\|^{2}-\|v\|^{2}\right]+\thetaP(u)\\
 & \leq\frac{m_{1}^{+}}{2}\left[\aM\|u-\XtM\|^{2}+\AM\|\YM-\XtM\|^{2}\right]+\frac{\varepsilon A_{+}}{\lam}+\thetaM(u)\leq m_{1}^{+}\left(\aM D_{h}^{2}+2D_{\Omega}^{2}\right)+\frac{\varepsilon A_{+}}{\lam}+\thetaM(u).
\end{align*}
Re-arranging the above terms and using \eqref{eq:d_aicg_lam_restr}
together with the first inequality in \eqref{eq:d_aicg_main_ineq},
we conclude that 
\begin{align*}
 & m_{1}^{+}\left(\aM D_{h}^{2}+2D_{\Omega}^{2}\right)+\thetaM(u)-\thetaP(u)\geq\frac{\AP}{2\lam}\left[(1-\lam M_{1})\|\YP-\XtM\|^{2}-\|v\|^{2}-2\varepsilon\right]\\
 & \geq\frac{\AP(1-\lam M_{1}-\theta^{2})}{2\lam}\|\YP-\XtM\|^{2}\geq\frac{\AP}{4\lam}\|\YP-\XtM\|^{2}.
\end{align*}
\end{proof}
The following result describes some important technical bounds obtained
by summing \eqref{eq:descent_d_aicg} for two different choices of
$u$ (possibly changing with $i$) from $i=1$ to $k$.
\begin{prop}
\label{prop:sum_d_aicg_descent} Let $\Delta_{\phi}^{0}$ and $d_{0}$
be as in \eqref{eq:d_aicg_diam} and define 
\begin{gather}
S_{k}:=\frac{1}{4\lam}\sum_{i=1}^{k}A_{i}\|\aicgY i-\aicgXTilde{i-1}\|^{2}.\label{eq:S_k_def}
\end{gather}
Then, the following statements hold: 
\begin{itemize}
\item[(a)] $S_{k}={\cal O}_{1}(k^{2}[m_{1}^{+}D_{h}^{2}+\Delta_{\phi}^{0}]+k[m_{1}^{+}+1/\lam]D_{\Omega}^{2});$ 
\item[(b)] if $f_{2}$ is convex, then $S_{k}={\cal O}_{1}(k^{2}m_{1}^{+}D_{h}^{2}+km_{1}^{+}D_{\Omega}^{2}+d_{0}^{2}/\lam).$ 
\end{itemize}
\end{prop}

\begin{proof}
(a) Let $\Delta_{1}(\cdot;\cdot,\cdot)$ be defined as in \eqref{eq:Delta_def}
with $(\psi_{s},\psi_{n})$ given by \eqref{eq:gen_acg_inputs}. Using
\eqref{eq:theta_def}, the fact that $\aicgX i,\aicgY i\in\Omega$,
the fact that $A_{i}$ is nonnegative and increasing, and the definitions
of $\theta_{i}$ and $D_{\Omega}$ in \eqref{eq:theta_def} and \eqref{eq:d_aicg_diam},
respectively, we have that 
\begin{align}
 & \sum_{i=1}^{k}\left[\theta_{i-1}(\aicgYMin i)-\theta_{i}(\aicgYMin i)\right]\leq\sum_{i=1}^{k}A_{i-1}\left[\phi(\aicgYMin{i-1})-\phi(\aicgYMin i)\right]+\frac{1}{2\lam}\sum_{i=1}^{k}\|\aicgYMin i-\aicgX{i-1}\|^{2}\nonumber \\
 & \leq A_{k}\sum_{i=1}^{k}\left[\phi(\aicgYMin{i-1})-\phi(\aicgYMin i)\right]+\frac{k}{2\lam}D_{\Omega}^{2}\leq A_{k}\left[\phi(\aicgYMin 0)-\phi_{*}\right]+\frac{k}{2\lam}D_{\Omega}^{2}.\label{eq:ncvx_theta_bd}
\end{align}
Moreover, noting Lemma~\ref{lem:d_aicg_Delta_props}(d) and using
Proposition~\ref{prop:descent_d_aicg} with $u=y_{i}$, we conclude
that \eqref{eq:descent_d_aicg} holds with $u=y_{i}$ for every $1\leq i\leq k$.
Summing these $k$ inequalities and using \eqref{eq:ncvx_theta_bd},
the definition of $S_{k}$ in \eqref{eq:S_k_def}, and Lemma~\ref{lem:A_k_props}(b)
yields the desired conclusion.

(b) Assume now that $f_{2}$ is convex and let $\aicgYMin *$ be a
point such that $\phi(\aicgYMin *)=\phi_{*}$ and $\|\aicgYMin 0-\aicgYMin *\|=d_{0}$.
It then follows from Lemma~\ref{lem:d_aicg_Delta_props}(b) and Proposition~\ref{prop:gen_refinement}(d)
with $(\icgY{},v)=(\aicgY i,v_{i})$ that $\Delta_{1}(\aicgYMin *;\aicgY i,v_{i})\leq\varepsilon$
for every $1\leq i\leq k$. The conclusion now follows by using an
argument similar to the one in (a) but which instead sums \eqref{eq:descent_d_aicg}
with $u=\aicgYMin *$ from $i=1$ to $k$, and uses the fact that
\begin{align*}
\sum_{i=1}^{k}\left[\theta_{i-1}(\aicgYMin *)-\theta_{i}(\aicgYMin *)\right]=\theta_{0}(\aicgYMin *)-\theta_{k}(\aicgYMin *)\leq\frac{1}{2\lam}\|\aicgYMin 0-\aicgYMin *\|^{2}=\frac{d_{0}}{2\lam},
\end{align*}
where the inequality is due to the fact that $\theta_{k}(\aicgYMin *)\geq0$
(see \eqref{eq:theta_def}) and $A_{0}=0$. 
\end{proof}
We now establish the rate at which the residual $\|\hat{v}_{i}\|$
tends to 0.
\begin{prop}
\label{prop:gen_v_hat_rate_d_aicg}Let $S_{k}$ be as in \eqref{eq:S_k_def}.
Moreover, define the quantities 
\begin{gather}
\begin{gathered}L_{1,k}^{{\rm avg}}:=\frac{1}{k}\sum_{i=1}^{k}L_{1}(\aicgY i,\aicgXTilde{i-1}),\quad C_{\lam,k}^{{\rm avg}}:=\frac{1}{k}\sum_{i=1}^{k}C_{\lam}(\hat{\icgY{}}_{i},\aicgY i),\\
D_{k}^{{\rm erg}}:=L_{1,k}^{{\rm erg}}+\frac{\theta}{\lam}C_{\lam,k}^{{\rm erg}},\quad
8\sqrt{2}\left(\frac{2+\lam L_{1}+\theta\overline{C}_{\lam}}{\lam}\right),
\end{gathered}
\label{eq:d_avg_def}
\end{gather}
where $C_{\lam}(\cdot,\cdot)$ and $\overline{C}_{\lam}$ are as in
\eqref{eq:C_lam_fn_def} and \eqref{eq:ell_phi_C_bar_lam_def}, respectively.
Then, we have 
\begin{align*}
\min_{i\leq k}\|\hat{v}_{i}\| & ={\cal O}_{1}\left(\left[\sqrt{\lam}L_{1,k}^{{\rm avg}}+\frac{1+\theta C_{\lam,k}^{{\rm avg}}}{\sqrt{\lam}}\right]\left[\frac{S_{k}}{k^{3}}\right]^{1/2}\right)+\frac{\hat{\rho}}{2}.
\end{align*}
\end{prop}

\begin{proof}
Let $\ell=\left\lceil k/2\right\rceil $. Using Lemma~\ref{lem:spec_refine}
with $(z,w)=(\aicgY i,\XtM_{i-1})$ and the bounds $C_{\lam}(\cdot,\cdot)\leq\overline{C}_{\lam}$
and $L_{1}(\cdot,\cdot)\leq L_{1}$ we have that $\|\hat{v}_{i}\|\leq{\cal E}_{i}\|\aicgY i-\aicgXTilde{i-1}\|$,
for every $\ell\leq i\leq k$, where 
\[
{\cal E}_{i}=\frac{2+\lam L_{1}(\aicgY i,\aicgXTilde{i-1})+\theta C_{\lam}(\hat{\icgY{}}_{i},\aicgY i)}{\lam}\quad\forall i\geq1.
\]
As a consequence, using the definition of $S_{k}$ in \eqref{eq:S_k_def},
the definitions in \eqref{eq:d_avg_def}, Lemma~\ref{lem:p_norm_tech}
with $p=3/2$, $a_{i}={\cal E}_{i}/\sqrt{A_{i}}$, and $b_{i}=\sqrt{A_{i}}\|\aicgY i-\aicgXTilde{i-1}\|$
for $i\in\{\ell,...,k\}$, Lemma~\ref{lem:A_k_props}(b), and the
fact that $(k-\ell+1)\geq k/2$, yields 
\begin{align*}
 & \min_{\ell\leq i\leq k}\|\hat{v}_{i}\|\leq\min_{\ell\leq i\leq k}{\cal E}_{i}\|\aicgY i-\aicgXTilde{i-1}\|\\
 & \leq\frac{1}{(k-\ell+1)^{3/2}}\left(\sum_{i=\ell}^{k}\frac{{\cal E}_{i}}{\sqrt{A_{i}}}\right)\left(\sum_{i=\ell}^{k}A_{i}\|\aicgY i-\aicgXTilde{i-1}\|^{2}\right)^{1/2}\\
 & \leq\frac{2^{3/2}}{k^{3/2}}\left(\frac{2}{k}\sum_{i=1}^{k}{\cal E}_{i}\right)\left(4\lam S_{k}\right)^{1/2}={\cal O}_{1}\left(\left[\sqrt{\lam}L_{1,k}^{{\rm avg}}+\frac{1+\theta C_{\lam,k}^{{\rm avg}}}{\sqrt{\lam}}\right]\left[\frac{S_{k}}{k^{3}}\right]^{1/2}\right).
\end{align*}
\end{proof}
We are now ready to prove Theorem~\ref{thm:d_aicg_compl}. 
\begin{proof}[\textit{Proof of Theorem~\ref{thm:d_aicg_compl}}]
(a) This follows from Proposition~\ref{prop:gen_v_hat_rate_d_aicg},
Proposition~\ref{prop:sum_d_aicg_descent}(a), the fact that $C_{\lam}(\cdot,\cdot)\leq\overline{C}_{\lam}$
and $L_{f_{1}}(\cdot,\cdot)\leq L_{1}$, and the termination condition
in step~4.

\noindent (b) The fact that $(\hat{\icgY{}},\hat{v})=(\hat{\icgY{}}_{k},\hat{v}_{k})$
satisfies the inclusion of \eqref{eq:rho_approx_soln} follows from
Lemma~\ref{lem:spec_refine} with $(\icgY{},v,\acgX 0)=(\aicgY k,v_{k},\aicgXTilde{k-1})$.
The fact that $\|\hat{v}\|\leq\hat{\rho}$ follows from the stopping
condition in step~4.

(c) The fact that the method does not fail follows from Proposition~\ref{prop:acg_properties}(c).
The bound in \eqref{eq:d_aicg_cvx_outer_compl} follows from a similar
argument as in part (a) except that Proposition~\ref{prop:sum_d_aicg_descent}(a)
is replaced with Proposition~\ref{prop:sum_d_aicg_descent}(b). 
\end{proof}

\appendix
\noindent 

\section{Technical Bounds}

\label{app:subdiff}

The result below presents a basic property of the composite gradient step. 
\begin{prop}
\label{prop:basic_refinement}Let $h\in\cConv({\cal Z})$, $z\in\dom h$,
and $g$ be a differentiable function on $\dom h$ which satisfies
$g(u)-\ell_{g}(u;z)\leq L\|u-z\|^{2}/2$ for some $L\geq0$ and every
$u\in\dom g$. Moreover, define
\begin{gather*}
\hat{z}:=\argmin_{u}\left\{ \ell_{g}(u;z)+h(u)+\frac{L}{2}\|u-z\|^{2}\right\}.
\end{gather*}
Then, it holds that 
\begin{gather*}
\frac{L}{2}\|z-\hat z\|^{2}\leq (g+h)(z)-(g+h)(\hat{z}).
\end{gather*}
\end{prop}

\begin{proof}
Using the definition of $\hat{z}$, the fact that $\ell_{g}(\cdot;z)+h(\cdot)+L\|\cdot-z\|^{2}/2$
is $L$-strongly convex, and the assumed bound $g(u)-\ell_{g}(u;z)\leq L\|u-z\|^{2}/2$
at $u=\hat{z}$, we have 
\begin{align*}
(g+h)(z) & =\ell_{g}(z;z)+h(z)\geq\ell_{g}(\hat{z};z)+h(\hat{z})+ L\|\hat{z}-z\|^{2} \geq(g+h)(\hat{z}) + \frac{L}{2}\|\hat{z}-z\|^{2}.
\end{align*}
\end{proof}

\section{\label{app:r_acg}R-ACG Algorithm}

This section presents technical results related to the R-ACG algorithm.

The first set of results describes some basic properties of the generated
iterates. 
\begin{prop}
\label{prop:acg_key_props} If $\psi_{s}$ is $\mu$--strongly convex,
then the following statements hold: 
\begin{itemize}
\item[(a)] $\acgY j=\argmin_{u\in{\cal Z}}\left\{ B_{j}\Gamma_{j}(u)+\|u-\acgY 0\|^{2}/2\right\} $; 
\item[(b)] $\Gamma_{j}\leq\psi$ and $B_{j}\psi(\acgX j)\leq\inf_{u\in{\cal Z}}\left\{ B_{j}\Gamma_{j}(u)+\|u-\acgY 0\|^{2}/2\right\} $; 
\item[(c)] $\eta_{j}\geq0$ and $\acgU j\in\pt_{\eta_{j}}\left(\psi-\mu\|\cdot-\acgX j\|^{2}/2\right)(\acgX j)$; 
\item[(d)] it holds that 
\[
\left(\frac{1}{1+\mu B_{j}}\right)\|B_{j}\acgU j+\acgX j-\acgX 0\|^{2}+2B_{j}\eta_{j}\leq\|\acgX j-\acgX 0\|^{2}
\]
\end{itemize}
\end{prop}

\begin{proof}
(a) See \cite[Proposition 1]{MontSvaiter_fista}.

(b) See \cite[Proposition 1(b)]{MontSvaiter_fista}.

(c) The optimality of $\acgY j$ in part (a), the $\mu$-strong convexity
of $\Gamma_{j}$, and the definition of $\acgU j$ imply that 
\begin{align*}
\acgU j & =\frac{\acgY 0-\acgY j}{B_{j}}+\mu(\acgX j-\acgY j)\in\pt\left(\Gamma_{j}-\frac{\mu}{2}\|\cdot-\acgY j\|^{2}+\mu\left\langle \cdot,\acgY j-\acgX j\right\rangle \right)(\acgY j)\\
 & =\pt\left(\Gamma_{j}-\frac{\mu}{2}\|\cdot-\acgX j\|^{2}\right)(\acgY j).
\end{align*}
Using the above inclusion, the definition of $\eta_{j}$, the fact
that $\Gamma_{j}-\mu\|\cdot\|^{2}/2$ is affine, and part (b), we
now conclude that 
\begin{align*}
\psi(z)-\frac{\mu}{2}\|z-\acgX j\|^{2} & \geq\Gamma_{j}(z)-\frac{\mu}{2}\|z-\acgX j\|^{2}=\Gamma_{j}(\acgY j)-\frac{\mu}{2}\|\acgY j-\acgX j\|^{2}+\left\langle \acgU j,z-\acgY j\right\rangle \\
 & =\psi(\acgX j)+\left\langle \acgU j,z-\acgX j\right\rangle -\eta_{j},
\end{align*}
for every $z\in\dom\psi_{n}$, which is exactly the desired inclusion.
The fact that $\eta_{j}\geq0$ follows from the above inequality with
$z=\acgX j$.

(d) It follows from parts (a)--(b) and the definition of $\eta_{j}$
that 
\begin{align*}
\eta_{j} & \leq\Gamma_{j}(u)+\frac{1}{2B_{j}}\|u-\acgX 0\|^{2}-\psi(\acgX j)\\
 & =\frac{\mu}{2}\|\acgX j-\acgY j\|^{2}-\frac{1}{B_{j}}\left\langle \acgX 0-\acgY j,\acgX j-\acgY j\right\rangle +\frac{1}{2B_{j}}\|\acgY j-\acgX 0\|^{2}\\
 & =\frac{1}{2B_{j}}\|\acgX j-\acgX 0\|^{2}-\frac{1}{2B_{j}}(1+\mu B_{j})\|\acgX j-\acgY j\|^{2}\\
 & =\frac{1}{2B_{j}}\|\acgX j-\acgX 0\|^{2}-\frac{1}{2B_{j}(1+\mu B_{j})}\|B_{j}\acgU j+\acgX j-\acgX 0\|^{2}.
\end{align*}
Multiplying both sides of the above inequality by $2B_{j}$ yields
the desired conclusion. 
\end{proof}
The next result presents the general iteration complexity of the algorithm,
i.e. Proposition~\ref{prop:acg_properties}(a).
\begin{proof}[Proof of Proposition~\ref{prop:acg_properties}(a)]
Let $\ell$ be the first iteration where 
\begin{equation}
\min\left\{ \frac{B_{\ell}^{2}}{4(1+\mu B_{\ell})},\frac{B_{\ell}}{2}\right\} \geq K_{\theta}^{2}\label{eq:spectral_A_bd}
\end{equation}
and suppose that the R-ACG has not stopped with failure before iteration
$\ell$. We show that it must stop with success at the end of the
$\ell^{{\rm th}}$ iteration. Combining the triangle inequality, the
successful check in step~3 of the method, \eqref{eq:spectral_A_bd},
and the relation $(a+b)^{2}\leq2a^{2}+2b^{2}$ for all $a,b\in\r,$
we first have that 
\begin{align*}
 & \|\acgU{\ell}\|^{2}+2\eta_{\ell}\\
 & \leq\max\left\{ \frac{1+\mu B_{\ell}}{A_{\ell}^{2}},\frac{1}{2B_{\ell}}\right\} \left(\frac{1}{1+\mu B_{\ell}}\|B_{\ell}\acgU{\ell}\|^{2}+4B_{\ell}\eta_{\ell}\right)\\
 & \leq\max\left\{ \frac{1+\mu B_{\ell}}{B_{\ell}^{2}},\frac{1}{2B_{\ell}}\right\} \left(\frac{2}{1+\mu B_{\ell}}\|B_{\ell}\acgU{\ell}+\acgX{\ell}-\acgX 0\|^{2}+2\|\acgX{\ell}-\acgX 0\|^{2}+4B_{\ell}\eta_{\ell}\right)\\
 & \leq\max\left\{ \frac{4(1+\mu B_{\ell})}{B_{\ell}^{2}},\frac{2}{B_{\ell}}\right\} \|\acgX{\ell}-\acgX 0\|^{2}\leq\frac{1}{K_{\theta}^{2}}\|\acgX{\ell}-\acgX 0\|^{2}\leq\theta^{2}\|\acgX{\ell}-\acgX 0\|^{2},
\end{align*}
and hence the method must terminate at the $\ell^{{\rm th}}$ iteration.
We now bound $\ell$ based on the requirement in \eqref{eq:spectral_A_bd}.
Solving for the quadratic in $B_{\ell}$ in the first bound of \eqref{eq:spectral_A_bd},
it is easy to see that $B_{\ell}\geq4\mu K_{\theta}^{2}+2K_{\theta}$
implies \eqref{eq:spectral_A_bd}. On the other hand, for the second
condition in \eqref{eq:spectral_A_bd}, it is immediate that $B_{\ell}\geq2K_{\theta}^{2}$
implies \eqref{eq:spectral_A_bd}. In view of \eqref{eq:B_j_bd} and
the previous two bounds, it follows that
\[
B_{\ell}\geq\frac{1}{L}\left(1+\sqrt{\frac{\mu}{4L}}\right)^{2(\ell-1)}\geq2K_{\theta}(1+2\mu K_{\theta}^{2})
\]
implies \eqref{eq:spectral_A_bd}. Using the bound $\log(1+t)\geq t/(1+t)$
for $t\geq0$ and the above bound on $\ell$, it is straightforward
to see that $\ell$ is on the same order of magnitude as in \eqref{eq:r_acg_total_compl}.
\end{proof}

\section{Refined ICG Points}

This appendix presents technical results related to the refined points
of the ICG methods.

The result below proves Lemma~\ref{lem:spec_refine} from the main
body of the paper.
\begin{proof}[Proof of Lemma~\ref{lem:spec_refine}]
(a) Using Proposition~\ref{prop:gen_refinement}(a), the definition
of $\hat{v}$, and the definitions of $\psi_{s}$ and $\psi_{n}$
in \eqref{eq:gen_acg_inputs}, we have that 
\begin{align*}
\hat{v} & \in\frac{1}{\lam}\left[\nabla\psi_{s}(\hat{\icgY{}})+\pt\psi_{n}(\hat{\icgY{}})+w-\icgY{}\right]+\nabla f_{1}(\hat{\icgY{}})-\nabla f_{1}(w)\\
 & =\frac{1}{\lam}\left[\lam\nabla f_{1}(w)+\lam f_{2}(\hat{\icgY{}})+(w-\icgY{})+\lam\pt h(\icgY{})\right]+\nabla f_{1}(\hat{\icgY{}})-\nabla f_{1}(w)\\
 & =\nabla f_{1}(\hat{\icgY{}})+\nabla f_{2}(\hat{\icgY{}})+\pt h(\hat{\icgY{}}),
\end{align*}
(b) Using assumption (A3), Proposition~\ref{prop:gen_refinement}(b),
the choice of $M$ in \eqref{eq:gen_acg_inputs}, and the fact that
$\Delta_{\mu}(\icgY r;\icgY{},v)\leq\varepsilon$, we first observe
that 
\begin{align}
 & \|\nabla f_{1}(\hat{\icgY{}})-\nabla f_{1}(\acgX 0)\|-L_{1}(\icgY{},\acgX 0)\|\icgY{}-\acgX 0\|\leq L_{1}(\icgY{},\hat{\icgY{}})\|\hat{\icgY{}}-\icgY{}\|\nonumber \\
 & \leq\frac{L_{1}(\icgY{},\hat{\icgY{}})\sqrt{2\Delta_{\mu}(\icgY r;\icgY{},v)}}{\sqrt{\lam M_{2}^{+}+1}}\leq\frac{\theta L_{1}(\icgY{},\hat{\icgY{}})}{\sqrt{\lam M_{2}^{+}+1}}\|\icgY{}-\acgX 0\|.\label{eq:deltaf1_bd}
\end{align}
Using now \eqref{eq:deltaf1_bd}, the choice of $M$ in \eqref{eq:gen_acg_inputs},
Proposition~\ref{prop:gen_refinement}(c) with $L(\cdot,\cdot)=\lam L_{2}(\cdot,\cdot)$,
the fact that $\sigma\leq1$, and the definition of $C_{\lam}(\cdot,\cdot)$,
we conclude that 
\begin{align*}
\|\hat{v}\| & \leq\frac{1}{\lam}\|v_{r}\|+\frac{1}{\lam}\|\icgY{}-\acgX 0\|+\|\nabla f_{1}(\hat{\icgY{}})-\nabla f_{1}(\acgX 0)\|\\
 & \leq\left[L_{1}(\icgY{},\acgX 0)+\frac{1+\theta}{\lam}+\frac{\theta\left[\lam M_{2}^{+}+1+\lam L_{1}(\icgY{},\hat{\icgY{}})+\lam L_{2}(\icgY{},\hat{\icgY{}})\right]}{\lam\sqrt{\lam M_{2}^{+}+1}}\right]\|\icgY{}-\acgX 0\|\\
 & \leq\left[L_{1}(\icgY{},\acgX 0)+\frac{2+\theta C_{\lam}(\icgY{},\hat{\icgY{}})}{\lam}\right]\|\icgY{}-\acgX 0\|.
\end{align*}
\end{proof}

\section{Spectral Functions}

\label{app:spectral} This section presents some results about spectral
functions as well as the proof of Propositions~\ref{prop:acg_implementation}.
It is assumed that the reader is familiar with the key quantities
given in Subsection~\ref{subsec:spectral_exploit} (e.g., see \eqref{eq:SVD_quants}
and \eqref{eq:vec_acg_inputs}).

We first state two well-known results \cite{lewis1995convex,Beck2017}
about spectral functions.
\begin{lem}
\label{lem:spec_prop}Let $\Psi=\Psi^{{\cal V}}\circ\sigma$ for some
absolutely symmetric function $\Psi^{{\cal V}}:\r^{r}\mapsto\r$.
Then, the following properties hold: 
\begin{itemize}
\item[(a)] $\Psi^{*}=(\Psi^{{\cal V}}\circ\sigma)^{*}=(\Psi^{{\cal V}})^{*}\circ\sigma$; 
\item[(b)] $\nabla\Psi=(\nabla\Psi^{{\cal V}})\circ\sigma$; 
\end{itemize}
\end{lem}

\begin{lem}
\label{lem:spec_prox} Let $(\Psi,\Psi^{{\cal V}})$ be as in Lemma~\ref{lem:spec_prop},
the pair $(S,\acgMatX{})\in{\cal Z}\times\dom\Psi$ be fixed, and
the decomposition $S=P[\dg\sigma(S)]Q^{*}$ be an SVD of $S$, for
some $(P,Q)\in{\cal U}^{m}\times{\cal U}^{n}$. If $\Psi\in\cConv\r^{m\times n}$
and $\Psi^{{\cal V}}\in\cConv\r^{r}$, then for every $M>0$, we have
\[
S\in\pt\left(\Psi+\frac{M}{2}\|\cdot\|_{F}^{2}\right)(\acgMatX{})\iff\begin{cases}
\sigma(S)\in\pt\left(\Psi^{{\cal V}}+\frac{M}{2}\|\cdot\|^{2}\right)(\sigma(\acgMatX{})),\\
\acgMatX{}=P[\dg\sigma(\acgMatX{})]Q^{*}.
\end{cases}
\]
\end{lem}

We now present a new result about spectral functions. 
\begin{thm}
\label{thm:spectral_approx_subdiff}Let $(\Psi,\Psi^{{\cal V}})$
be as in Lemma~\ref{lem:spec_prop} and the point $\acgMatX{}\in\r^{m\times n}$
be such that $\sigma(\acgMatX{})\in\dom\Psi^{{\cal V}}$. Then for
every $\varepsilon\geq0$, we have $S\in\pt_{\varepsilon}\Psi(\acgMatX{})$
if and only if $\sigma(S)\in\pt_{\varepsilon(S)}\Psi^{{\cal V}}(\sigma(\acgMatX{}))$,
where 
\begin{equation}
\varepsilon(S):=\varepsilon-\left[\left\langle \sigma(\acgMatX{}),\sigma(S)\right\rangle -\left\langle \acgMatX{},S\right\rangle \right]\geq0.\label{eq:eps_s_spectral}
\end{equation}
Moreover, if $S$ and $Z$ have a simultaneous SVD, then $\varepsilon(S)=\varepsilon$. 
\end{thm}

\begin{proof}
Using Lemma~\ref{lem:spec_prop}(a), \eqref{eq:eps_s_spectral},
and the well-known fact that $S\in\pt_{\varepsilon}\Psi(\acgMatX{})$
if and only if $\varepsilon\geq\Psi(\acgMatX{})+\Psi^{*}(S)-\left\langle \acgMatX{},S\right\rangle $,
we have that $S\in\pt_{\varepsilon}\Psi(\acgMatX{})$ if and only
if 
\begin{align*}
\varepsilon(S) & =\varepsilon-\left[\left\langle \sigma(\acgMatX{}),\sigma(S)\right\rangle -\left\langle \acgMatX{},S\right\rangle \right]\\
 & \geq\Psi(\acgMatX{})+\Psi^{*}(S)-\left\langle \acgMatX{},S\right\rangle -\left[\left\langle \sigma(\acgMatX{}),\sigma(S)\right\rangle -\left\langle \acgMatX{},S\right\rangle \right]\\
 & =\Psi^{{\cal V}}(\sigma(\acgMatX{}))+(\Psi^{{\cal V}})^{*}(\sigma(S))-\left\langle \sigma(\acgMatX{}),\sigma(S)\right\rangle ,
\end{align*}
or, equivalently, $\sigma(S)\in\pt_{\varepsilon(S)}\Psi^{{\cal V}}(\sigma(\acgMatX{}))$
and $\varepsilon(S)\geq0$. To show that the existence of a simultaneous
SVD of $S$ and $Z$ implies $\varepsilon(S)=\varepsilon$ it suffices
to show that $\inner{\sigma(S)}{\sigma(\acgMatX{})}=\inner S{\acgMatX{}}$.
Indeed, if $S=P[\dg\sigma(S)]Q^{*}$ and $\acgMatX{}=P[\dg\sigma(\acgMatX{})]Q^{*}$,
for some $(P,Q)\in{\cal U}^{m}\times{\cal U}^{n}$, then we have 
\[
\inner S{\acgMatX{}}=\inner{\dg\sigma(S)}{P^{*}P[\dg\sigma(\acgMatX{})]Q^{*}Q}=\inner{\dg\sigma(S)}{\dg\sigma(\acgMatX{})}=\inner{\sigma(S)}{\sigma(\acgMatX{})}.
\]
\end{proof}

\section*{Acknowledgments}
 The authors would like to thank the two anonymous referees and the associate editor for their insightful comments on earlier drafts of this paper.

\bibliographystyle{plain}
\bibliography{Proxacc_ref}

\end{document}